\documentclass{amsart}
\usepackage{eurosym}
\usepackage{amsmath}
\usepackage{amssymb}
\usepackage{amsfonts}
\usepackage{graphicx}

\newtheorem{theorem}{Theorem}[section]
\theoremstyle{plain}

\newtheorem{lemma}[theorem]{Lemma}

\newtheorem{proposition}[theorem]{Proposition}

\numberwithin{equation}{section}
\input{tcilatex}

\begin{document}
\title[Nodal solutions to a semilinear elliptic problem on a manifold]{%
Multiplicity and asymptotic profile of $2$-nodal solutions to a semilinear
elliptic problem on a Riemannian manifold}
\author{M\'{o}nica Clapp}
\address{Instituto de Matem\'{a}ticas, Universidad Nacional Aut\'{o}noma de M%
\'{e}xico, Circuito Exterior CU, 04510 M\'{e}xico DF, Mexico}
\email{monica.clapp@im.unam.mx}
\author{Anna Maria Micheletti}
\address{Dipartimento di Matematica Applicata, Universit\`{a} di Pisa, Via
Buonarroti 1/c 56127, Pisa, Italy}
\email{a.micheletti@dma.unipi.it}
\thanks{Research supported by CONACYT grant 129847 and UNAM-DGAPA-PAPIIT
grant IN106612 (Mexico).}
\date{December, 2012}

\begin{abstract}
We establish a lower bound for the number of sign changing solutions with
precisely two nodal domains to the singularly perturbed nonlinear elliptic
equation $-\varepsilon ^{2}\Delta _{g}u+u=|u|^{p-2}u$ on an $n$-dimensional
Riemannian manifold $M,$ $p\in (2,2^{\ast }),$ in terms of the cup-length of
the configuration space of $M.$ We give a precise description of the
asymptotic profile of these solutions as $\varepsilon \rightarrow 0.$ We
also provide new estimates for the cup-length of the configuration space of $%
M.\medskip $

\noindent \textsc{2010 Mathematics Subject Classification.} Primary 58J05;
Secondary 35B25, 35J61, 35R01, 55R12, 55R80.\medskip

\noindent \textsc{Key words and phrases.} Nonlinear elliptic equation,
Riemannian manifold, singularly perturbed problem, sign changing solutions,
transfer, cup-length of configuration spaces.
\end{abstract}

\maketitle

\section{Introduction}

Let $(M,g)$ be a compact connected Riemannian manifold, without boundary, of
class $\mathcal{C}^{k}$ with $k\geq 1$. Let $n\geq 2$ be the dimension of $M$%
. We consider the following problem 
\begin{equation}
\left\{ 
\begin{array}{l}
-\varepsilon ^{2}\Delta _{g}u+u=|u|^{p-2}u, \\ 
u\in H_{g}^{1}(M),%
\end{array}%
\right.  \label{eq:Peps}
\end{equation}%
with $2<p<2^{\ast },$ where $2^{\ast }:=\frac{2n}{n-2}$ if $n>2$ and $%
2^{\ast }=\infty $ if $n=2$. The space $H_{g}^{1}(M)$ is the completion of $%
\mathcal{C}^{\infty }(M)$ with respect to the norm defined by $\Vert u\Vert
_{g}^{2}=\int_{M}(\left\vert \nabla _{g}u\right\vert ^{2}+u^{2})d\mu _{g}$.

This problem resembles the Neumann problem on a flat domain, which has been
widely studied in literature, see e.g. \cite%
{DY,DFW,G,GWW,Li,NT1,NT2,W1,WW05,WW}.

Existence, multiplicity and shape of positive solutions to (\ref{eq:Peps})
have been investigated in several papers. The existence of a mountain pass
solution was proved by Byeon and Park in \cite{BP05}, who also showed that
this solution has a spike which approaches a maximum point of the scalar
curvature as $\varepsilon \rightarrow 0$. Benci, Bonanno and Micheletti \cite%
{BBM07} showed that problem (\ref{eq:Peps}) has at least cat$(M)+1$ positive
solutions if $\varepsilon $ is small enough, where cat$(M)$ denotes the
Lusternik-Schnirelmann category of $M$. A similar result for more general
nonlinearities was obtained in \cite{V08}.\ In \cite{H09} N. Hirano gave a
lower bound for the number of positive solutions to (\ref{eq:Peps}) in terms
of the category of a set determined by the geometry of $M$. Solutions with
one peak which concentrates at a stable critical point of the scalar
curvature of $M$ as $\varepsilon \rightarrow 0$ were obtained by Micheletti
and Pistoia in \cite{MP09}, and in \cite{DMP09} Dancer, Micheletti and
Pistoia proved the existence solutions with $k$ peaks which concentrate at
an isolated minimum of the scalar curvature of $M$ as $\varepsilon
\rightarrow 0.$

Concerning sign changing solutions only few results are known. When the
scalar curvature of $M$ is not constant a solution with one positive peak
and one negative peak, concentrating at a maximum and a minimum of the
scalar curvature, was obtained by Micheletti and Pistoia in \cite{MP09b}.
Multiplicity of solutions which change sign exactly once was established by
Ghimenti and Micheletti in \cite{GM10} for Riemannian manifolds $M$ which
are invariant with respect to an orthogonal involution.

Here we provide a lower bound for the number of sign changing solutions with
precisely two nodal domains, without requiring any symmetry or geometric
assumptions on the manifold $M$, and we give a precise description of the
asymptotic profile of these solutions as $\varepsilon \rightarrow 0$.

In order to state our main result we introduce some notation. The limiting
problem to problem (\ref{eq:Peps}) as $\varepsilon \rightarrow 0$ is 
\begin{equation}
-\Delta u+u=|u|^{p-2}u,\text{\qquad }u\in H^{1}(\mathbb{R}^{n}).
\label{eq:a1}
\end{equation}%
It is well known that, up to translations, this problem has a unique
positive solution, which is spherically symmetric, and is usually denoted by 
$U\in H^{1}(\mathbb{R}^{n})$.

The exponential map $\exp :TM\rightarrow M,$ defined on the total space $TM$
of the tangent bundle of $M$, is a $\mathcal{C}^{\infty }$ map. Since $M$ is
compact, there exists $R>0$ such that $\exp _{q}:B(0,R)\rightarrow
B_{g}(q,R) $ is a diffeomorphism for every $q\in M$. Here $T_{q}M$ is
identified with $\mathbb{R}^{n}$, $B(0,R)$ is the ball of radius $R$ in $%
\mathbb{R}^{n}$ centered at $0,$ and $B_{g}(q,R)$ denotes the ball of radius 
$R$ in $M$ centered at $q$ with respect to the distance induced by the
Riemannian metric $g$.

Let $\chi _{R}\in \mathcal{C}^{\infty }(\mathbb{R}^{n})$ be a radial cut-off
function such that $\chi _{R}(z)=1$ if $\left\vert z\right\vert \leq R/2$, $%
\chi _{R}(z)=0$ if $\left\vert z\right\vert \geq R$, and $\left\vert \nabla
\chi _{R}(z)\right\vert <2/R$ and $\left\vert \nabla ^{2}\chi
_{R}(z)\right\vert <2/R^{2}$ for all $z\in \mathbb{R}^{n}.$ For $\xi \in M$
and $\varepsilon >0$ we define $W_{\varepsilon ,\xi }\in H_{g}^{1}(M)$ by 
\begin{equation}
W_{\varepsilon ,\xi }(x)=\left\{ 
\begin{array}{ll}
U\left( \frac{\exp _{\xi }^{-1}(x)}{\varepsilon }\right) \chi _{R}\left( 
\frac{\exp _{\xi }^{-1}(x)}{\varepsilon }\right) & \text{if }x\in B_{g}(\xi
,\varepsilon R), \\ 
0 & \text{otherwise}.%
\end{array}%
\right.  \label{eq:Weps}
\end{equation}

Let $F(M):=\{(x,y)\in M\times M:x\neq y\}$. The quotient space $C(M)$ of $%
F(M)$ obtained by identifying $(x,y)$ with $(y,x)$ is called the
configuration space of $M.$

We write $\mathcal{H}^{\ast }$ for singular cohomology with coefficients in $%
\mathbb{Z}/2$. Recall that the cup-length of a topological space $X$ is the
smallest integer $k\geq 1$ such that the cup-product of any $k$ cohomology
classes in $\widetilde{\mathcal{H}}^{\ast }(X)$ is zero, where $\widetilde{%
\mathcal{H}}^{\ast }$ is reduced cohomology. We denote it by cupl$X.$

We are ready to state our main result.

\begin{theorem}
\label{thm:A}There exists $\varepsilon _{0}>0$ such that for every $%
\varepsilon \in (0,\varepsilon _{0})$ problem \emph{(\ref{eq:Peps})} has at
least \emph{cupl\thinspace }$C(M)$ pairs of sign changing solutions $\pm
u_{\varepsilon }$ with the following properties:

\begin{enumerate}
\item[\emph{(a)}] $u_{\varepsilon }$ has a unique local maximum point $%
Q_{\varepsilon }$ and a unique local minimum point $q_{\varepsilon }$ on $M$.

\item[\emph{(b)}] For any fixed $T>0,$ 
\begin{align*}
\lim_{\varepsilon \rightarrow 0}\Vert u_{\varepsilon }(\exp _{Q_{\varepsilon
}}(\varepsilon z))-U(z)\Vert _{\mathcal{C}^{2}(B(0,T))}& =0, \\
\lim_{\varepsilon \rightarrow 0}\Vert u_{\varepsilon }(\exp _{q_{\varepsilon
}}(\varepsilon z))+U(z)\Vert _{\mathcal{C}^{2}(B(0,T))}& =0.
\end{align*}

\item[\emph{(c)}] Moreover, 
\begin{align*}
{\sup_{\xi \in M\smallsetminus B_{g}(Q_{\varepsilon },\varepsilon
T)}u_{\varepsilon }(\xi )}& <ce^{-\mu T}+\sigma _{1}(\varepsilon ), \\
\inf_{\xi \in M\smallsetminus B_{g}(q_{\varepsilon },\varepsilon
T)}u_{\varepsilon }(\xi )& >{-ce^{-\mu T}+\sigma }_{2}{(\varepsilon ),}
\end{align*}%
for some positive constants $c,\mu $ and for some functions $\sigma
_{1},\sigma _{2}$ which go to zero as $\varepsilon $ goes to zero.

\item[\emph{(d)}] The function $\Phi _{\varepsilon }$ given by%
\begin{equation*}
u_{\varepsilon }=W_{\varepsilon ,Q_{\varepsilon }}-W_{\varepsilon
,q_{\varepsilon }}+\Phi _{\varepsilon }
\end{equation*}%
is such that $\Vert \Phi _{\varepsilon }\Vert _{\varepsilon }\rightarrow 0$
as $\varepsilon \rightarrow 0$, with $\Vert \cdot \Vert _{\varepsilon }$ as
in \emph{(\ref{normepsilon}) }below.
\end{enumerate}
\end{theorem}

We shall see that%
\begin{equation*}
\text{cupl\thinspace }C(M)\geq n+1.
\end{equation*}%
However, this estimate can be improved in many cases. We prove the following
result.

\begin{theorem}
\label{thm:cuplength}If $\mathcal{H}^{i}(M)=0$ for all $0<i<m$ and if there
are $k$ cohomology classes $\zeta _{1},\ldots ,\zeta _{k}\in \mathcal{H}%
^{m}(M)$ whose cup-product is nontrivial, then%
\begin{equation*}
\emph{cupl\,}C(M)\geq k+n.
\end{equation*}
\end{theorem}

For example, if $M$ is homeomorphic to an $n$-torus $\mathbb{S}^{1}\times
\cdots \times \mathbb{S}^{1}$ (with $n$ factors) then cupl$\,C(M)=2n.$

Unlike the case of positive solutions or the symmetric case considered in 
\cite{GM10} where one looks for solutions on some Nehari manifold, there is
no natural constraint for sign changing solutions to problem (\ref{eq:Peps}%
). Our approach is based on some ideas introduced in \cite{BCW07}. We
exhibit a set $\mathcal{Z}_{\varepsilon }$ of sign changing functions which
is positively invariant under the negative gradient flow of the energy
functional associated to problem (\ref{eq:Peps}). This set does not have an
explicit description, so there is no way of defining a map from $C(M)$ into
it. However, using Dold's fixed point transfer \cite{d}, one can obtain a
homomorphism at the cohomological level. A careful study of the barycenter
map introduced in \cite{BBM07} allows us to establish a lower bound for the
equivariant Lusternik-Schnirelmann of low energy sublevel sets of $\mathcal{Z%
}_{\varepsilon }$ and, hence, for the number of sign changing solutions.

In contrast to the Lyapunov-Schmidt reduction method, applied for example in 
\cite{MP09b}, this method does not provide information on the asymptotic
profile of the solutions obtained, as $\varepsilon \rightarrow 0.$ We carry
out a careful analysis in order to show that they have the properties
described in Theorem \ref{thm:A}.

Finally we wish to point out that, although configuration spaces have been
widely studied, not much seems to be known about the multiplicative
structure of their cohomology. So we believe that Theorem \ref{thm:cuplength}
has an interest of its own. A similar estimate for the cup-length of the
configuration space of an open subset of $\mathbb{R}^{n}$ has been given in 
\cite{BCW07}.

The paper is organized as follows: In section \ref{sec:Sezione-1} we discuss
the variational setting for sign changing solutions to problem (\ref{eq:Peps}%
). In section \ref{sec:Sezione-2} we prove the multiplicity statement in
Theorem \ref{thm:A},\ and in section \ref{sec:Sezione-3} we prove that the
solutions have properties \emph{(a)-(d)}. Section \ref{sec:Sezione-5}\ is
devoted to the proof of Theorem \ref{thm:cuplength}.

\section{The variational setting}

\label{sec:Sezione-1}For $\varepsilon >0$ we take 
\begin{eqnarray}
(u,v)_{\varepsilon } &:&=\frac{1}{\varepsilon ^{n}}\int_{M}\left(
\varepsilon ^{2}\left\langle \nabla _{g}u,\nabla _{g}v\right\rangle
+uv\right) d\mu _{g},  \notag \\
\left\Vert u\right\Vert _{\varepsilon }^{2} &:&=\frac{1}{\varepsilon ^{n}}%
\int_{M}\left( \varepsilon ^{2}|\nabla _{g}u|^{2}+u^{2}\right) d\mu _{g}
\label{normepsilon}
\end{eqnarray}%
as the scalar product and the corresponding norm in $H_{g}^{1}(M)$. For any $%
u\in L_{g}^{p}(M)$ we put 
\begin{equation*}
\left\vert u\right\vert _{p,\varepsilon }:=\left( \frac{1}{\varepsilon ^{n}}%
\int_{M}|u|^{p}d\mu _{g}\right) ^{\frac{1}{p}},\qquad |u|_{p,g}:=\left(
\int_{M}|u|^{p}d\mu _{g}\right) ^{\frac{1}{p}},
\end{equation*}%
and define%
\begin{equation*}
S_{\varepsilon }:=\inf \left\{ \frac{\left\Vert u\right\Vert _{\varepsilon
}^{2}}{\left\vert u\right\vert _{p,\varepsilon }^{2}}:u\in H_{g}^{1}(M),%
\text{ }u\neq 0\right\} .
\end{equation*}

The solutions of (\ref{eq:Peps}) are the critical points of the functional 
\begin{equation*}
J_{\varepsilon }:H_{g}^{1}(M)\rightarrow \mathbb{R},\qquad J_{\varepsilon
}(u)=\frac{1}{2}\left\Vert u\right\Vert _{\varepsilon }^{2}-\frac{1}{p}%
\left\vert u\right\vert _{p,\varepsilon }^{p}.
\end{equation*}%
Any solution $u\neq 0$ of (\ref{eq:Peps}) lies on the Nehari manifold 
\begin{align*}
\mathcal{N}_{\varepsilon }:& =\left\{ u\in H_{g}^{1}(M)\smallsetminus
\left\{ 0\right\} :J_{\varepsilon }^{\prime }(u)u=0\right\} \\
& =\{u\in H_{g}^{1}(M)\smallsetminus \left\{ 0\right\} :\left\Vert
u\right\Vert _{\varepsilon }^{2}=\left\vert u\right\vert _{p,\varepsilon
}^{p}\},
\end{align*}%
which is a $\mathcal{C}^{2}$-manifold diffeomorphic to the unit sphere in $%
H_{g}^{1}(M)$. Any sign changing solution of (\ref{eq:Peps}) belongs to the
set 
\begin{equation*}
\mathcal{E}_{\varepsilon }:=\left\{ u\in H_{g}^{1}(M):u^{+},u^{-}\in 
\mathcal{N}_{\varepsilon }\right\} \subset \mathcal{N}_{\varepsilon }.
\end{equation*}%
Here $u^{+}:=\max \{u,0\}$ and $u^{-}:=\min \{u,0\}$. We set 
\begin{equation*}
c_{\varepsilon }:=\inf_{\mathcal{N}_{\varepsilon }}J_{\varepsilon }.
\end{equation*}%
It is easily checked that%
\begin{equation*}
c_{\varepsilon }=\frac{p-2}{2p}S_{\varepsilon }^{p/(p-2)}.
\end{equation*}

Analogously, we consider the functional 
\begin{equation*}
J_{\infty }:H^{1}(\mathbb{R}^{n})\rightarrow \mathbb{R},\text{\qquad }%
J_{\infty }(v)=\frac{1}{2}\int_{\mathbb{R}^{n}}\left( |\nabla
v|^{2}+v^{2}\right) dx-\frac{1}{p}\int_{\mathbb{R}^{n}}|v|^{p}dx.
\end{equation*}%
associated to the limit problem (\ref{eq:a1}). Any solution $v\neq 0$ to
this problem lies on the Nehari manifold 
\begin{equation*}
\mathcal{N}_{\infty }:=\left\{ u\in H^{1}(\mathbb{R}^{n})\smallsetminus
\left\{ 0\right\} :J_{\infty }^{\prime }(u)u=0\right\} .
\end{equation*}%
We set 
\begin{equation*}
c_{\infty }:=\inf_{\mathcal{N}_{\infty }}J_{\infty }.
\end{equation*}%
It is well known that there exists a unique positive solution $U\in H^{1}(%
\mathbb{R}^{n})$ to the limit problem (\ref{eq:a1}) which is spherically
symmetric with respect to the origin, and that $J_{\infty }(U)=c_{\infty }$.
Moreover, 
\begin{equation}
\lim\limits_{\left\vert x\right\vert \rightarrow \infty }U(x)\left\vert
x\right\vert ^{\frac{n-1}{2}}\exp \left\vert x\right\vert =b>0.
\label{asympU}
\end{equation}%
Setting ${U_{\varepsilon }(x):=U\left( \frac{x}{\varepsilon }\right) }$ we
have that 
\begin{equation*}
-\varepsilon ^{2}\Delta U_{\varepsilon }+U_{\varepsilon }=U_{\varepsilon
}^{p-1}.
\end{equation*}%
We recall the following property of the infima $c_{\varepsilon }$.

\begin{lemma}
\label{lem:lemma1}${\lim_{\varepsilon \rightarrow 0}}${$c_{\varepsilon
}=c_{\infty }$}.
\end{lemma}

\begin{proof}
See \cite{BBM07,BP05}.
\end{proof}

We consider the negative gradient flow $\varphi_{\varepsilon}:\mathcal{G}%
_{\varepsilon}\rightarrow H_{g}^{1}(M)$ of $J_{\varepsilon}$ defined by 
\begin{equation*}
\left\{ 
\begin{array}{l}
\frac{d}{dt}\varphi_{\varepsilon}(t,u)=-\nabla_{\varepsilon}J_{\varepsilon
}(\varphi_{\varepsilon}(t,u)), \\ 
\quad\varphi_{\varepsilon}(0,u)=u.%
\end{array}
\right.
\end{equation*}
Here $\nabla_{\varepsilon}J_{\varepsilon}$ is the gradient of $%
J_{\varepsilon }$ with respect to the scalar product $(\cdot,\cdot)_{%
\varepsilon}$ and $\mathcal{G}_{\varepsilon}=\left\{ (t,u):u\in
H_{g}^{1}(M),\ 0\leq t\leq T_{\varepsilon}(u)\right\} $, where $%
T_{\varepsilon}(u)\in(0,+\infty)$ is the maximal existence time for $%
\varphi_{\varepsilon}(t,u)$.

A subset $\mathcal{D}$ of $H_{g}^{1}(M)$ is called strictly positively
invariant for the flow $\varphi _{\varepsilon }$ if $\varphi _{\varepsilon
}(t,u)\in $ int$(\mathcal{D})$ for all $u\in \mathcal{D}$ and $t\in
(0,T_{\varepsilon }(u))$, where int$(\mathcal{D})$ denotes the interior of $%
\mathcal{D}$ in $H_{g}^{1}(M)$. If $\mathcal{D}$ is strictly positive
invariant then the set 
\begin{equation*}
\mathcal{A}_{\varepsilon }(\mathcal{D}):=\left\{ u\in H_{g}^{1}(M):\varphi
_{\varepsilon }(t,u)\in \mathcal{D}\text{ for some }t\in (0,T_{\varepsilon
}(u))\right\}
\end{equation*}%
is an open subset of $H_{g}^{1}(M)$.

Let $\mathcal{P}=\left\{ u\in H_{g}^{1}(M)\ :\ u\geq 0\right\} $ be the
convex cone of nonnegative functions and let 
\begin{equation*}
\mathcal{B}(\varepsilon ,\pm \mathcal{P}):=\left\{ u\in H_{g}^{1}(M):\text{%
dist}_{\varepsilon }(u,\pm \mathcal{P})\leq \frac{1}{2}S_{\varepsilon
}^{p/2(p-2)}\right\} ,
\end{equation*}%
where {dist$_{\varepsilon }(u,\mathcal{P}$}${):=\min_{v\in \mathcal{P}}\Vert
u-v\Vert _{\varepsilon }}$, ${\ }${dist$_{\varepsilon }(u,$}${-}${$\mathcal{P%
}$}${):=\min_{v\in -\mathcal{P}}\Vert u-v\Vert _{\varepsilon }}$.

\begin{lemma}
\label{lem:3}The following statements hold true:

\begin{enumerate}
\item[\emph{(a)}] $\left( \mathcal{B}(\varepsilon ,\mathcal{P})\cup \mathcal{%
B}\left( \varepsilon ,-\mathcal{P}\right) \right) \cap \mathcal{E}%
_{\varepsilon }=\emptyset $. Hence, if $u\in \mathcal{B}\left( \varepsilon ,%
\mathcal{P}\right) \cup \mathcal{B}\left( \varepsilon ,-\mathcal{P}\right) $
is a solution of \emph{(\ref{eq:Peps})} then $u$ does not change sign.

\item[\emph{(b)}] $\mathcal{B}\left( \varepsilon,\pm\mathcal{P}\right) $ is
strictly positively invariant for the flow $\varphi_{\varepsilon}$.
\end{enumerate}
\end{lemma}

\begin{proof}
The proof is analogous to that of Proposition 3.1 in \cite{BCW07}. We sketch
it here for the reader's convenience. Note first that%
\begin{equation}
\left\vert u^{-}\right\vert _{p,\varepsilon }=\min_{v\in \mathcal{P}%
}\left\vert u-v\right\vert _{p,\varepsilon }\leq S_{\varepsilon
}^{-1/2}\min_{v\in \mathcal{P}}\left\Vert u-v\right\Vert _{\varepsilon
}=S_{\varepsilon }^{-1/2}\text{dist}_{\varepsilon }(u,\mathcal{P}).
\label{eq:1}
\end{equation}%
So, if $u\in \mathcal{E}_{\varepsilon }\cap \mathcal{B}\left( \varepsilon ,%
\mathcal{P}\right) ,$ then 
\begin{equation*}
0<S_{\varepsilon }^{p/(p-2)}\leq \left\Vert u^{-}\right\Vert _{\varepsilon
}^{2}=\left\vert u^{-}\right\vert _{p,\varepsilon }^{p}\leq S_{\varepsilon
}^{-p/2}\text{dist}_{\varepsilon }(u,\mathcal{P})^{p}\leq \frac{1}{2^{p}}%
S_{\varepsilon }^{p/(p-2)}.
\end{equation*}%
which is a contradiction. Hence, $\mathcal{E}_{\varepsilon }\cap \mathcal{B}%
\left( \varepsilon ,\mathcal{P}\right) =\emptyset .$ Similarly, $\mathcal{E}%
_{\varepsilon }\cap \mathcal{B}\left( \varepsilon ,-\mathcal{P}\right)
=\emptyset .$ This proves \emph{(a)}.\newline
Next, we prove assertion \emph{(b)} for $\mathcal{B}\left( \varepsilon ,%
\mathcal{P}\right) $. A similar argument goes through for $\mathcal{B}\left(
\varepsilon ,-\mathcal{P}\right) $.\newline
The gradient $\nabla _{\varepsilon }J_{\varepsilon }$ with respect to the
scalar product $(\cdot ,\cdot )_{\varepsilon }$ is given by $\nabla
_{\varepsilon }J_{\varepsilon }=Id-K_{\varepsilon },$ where $K_{\varepsilon
}\in H_{g}^{1}(M)$ is defined by 
\begin{equation*}
(K_{\varepsilon }(u),v)_{\varepsilon }=\frac{1}{\varepsilon ^{n}}%
\int_{M}|u|^{p-2}uv\text{ }d\mu _{g}\qquad \forall v\in H_{g}^{1}(M).
\end{equation*}%
Using (\ref{eq:1}) we obtain 
\begin{align*}
\text{dist}_{\varepsilon }(K_{\varepsilon }(u),\mathcal{P})\Vert
K_{\varepsilon }(u)^{-}\Vert _{\varepsilon }& \leq \Vert K_{\varepsilon
}(u)^{-}\Vert _{\varepsilon }^{2}=(K_{\varepsilon }(u),K_{\varepsilon
}(u)^{-})_{\varepsilon }=\frac{1}{\varepsilon ^{n}}\int_{M}|u|^{p-2}uK_{%
\varepsilon }(u)^{-}d\mu _{g} \\
& =\frac{1}{\varepsilon ^{n}}\int_{M}\left( u^{+}\right)
^{p-1}K_{\varepsilon }(u)^{-}d\mu _{g}+\frac{1}{\varepsilon ^{n}}%
\int_{M}\left\vert u^{-}\right\vert ^{p-2}u^{-}K_{\varepsilon }(u)^{-}d\mu
_{g} \\
& \leq \frac{1}{\varepsilon ^{n}}\int_{M}\left\vert u^{-}\right\vert
^{p-2}u^{-}K_{\varepsilon }(u)^{-}d\mu _{g}\leq \left\vert K_{\varepsilon
}(u)^{-}\right\vert _{p,\varepsilon }\left\vert u^{-}\right\vert
_{p,\varepsilon }^{p-1} \\
& \leq S_{\varepsilon }^{-1/2}\left\Vert K_{\varepsilon }(u)^{-}\right\Vert
_{\varepsilon }S_{\varepsilon }^{-(p-1)/2}\text{dist}_{\varepsilon }(u,%
\mathcal{P})^{p-1}.
\end{align*}%
Therefore, dist$_{\varepsilon }(K_{\varepsilon }(u),\mathcal{P})\leq
S_{\varepsilon }^{-p/2}$dist$_{\varepsilon }(u,\mathcal{P})^{p-1}\leq \frac{1%
}{2^{p-1}}S^{p/2(p-2)}$ if $u\in \mathcal{B}\left( \varepsilon ,\mathcal{P}%
\right) $. Hence $K_{\varepsilon }(u)\in $ int$(\mathcal{B}\left(
\varepsilon ,\mathcal{P}\right) ).$ The rest of the argument is completely
analogous to that of Proposition 3.1 of \cite{BCW07}.
\end{proof}

For $\varepsilon >0$ we set 
\begin{align*}
\mathcal{D}_{\varepsilon }& :=\mathcal{B}\left( \varepsilon ,\mathcal{P}%
\right) \cup \mathcal{B}\left( \varepsilon ,-\mathcal{P}\right) \cup
J_{\varepsilon }^{0}, \\
\mathcal{Z}_{\varepsilon }& :=H_{g}^{1}(M)\smallsetminus \mathcal{A}%
_{\varepsilon }(\mathcal{D}_{\varepsilon }),
\end{align*}%
where $J_{\varepsilon }^{a}:=\left\{ u\in H_{g}^{1}(M):J_{\varepsilon
}(u)\leq a\right\} $ for $a\in \mathbb{R}.$ By Lemma \ref{lem:3} we have
that $\mathcal{D}_{\varepsilon }$ is strictly positively invariant for $%
\varphi _{\varepsilon }$. By Lemma \ref{lem:3} every function in $\mathcal{Z}%
_{\varepsilon }$ is sign changing and every sign changing solution to
problem (\ref{eq:Peps})\ lies in $\mathcal{Z}_{\varepsilon }.$\ We set 
\begin{equation*}
d_{\varepsilon }:=\inf_{\mathcal{Z}_{\varepsilon }}J_{\varepsilon }.
\end{equation*}%
The following version of Ekeland's variational principle holds true in this
setting.

\begin{lemma}
\label{ekeland}Given $\varepsilon >0$, $\delta >0$ and $u\in \mathcal{Z}%
_{\varepsilon }$ such that $J_{\varepsilon }(u)\leq d_{\varepsilon }+\delta
, $ there exists $v\in \mathcal{Z}_{\varepsilon }$ such that $J_{\varepsilon
}(v)\leq J_{\varepsilon }(u)$, $\Vert u-v\Vert _{\varepsilon }\leq \sqrt{%
\delta }$ and $\Vert \nabla _{\varepsilon }J_{\varepsilon }(v)\Vert
_{\varepsilon }\leq \sqrt{\delta }$.
\end{lemma}

\begin{proof}
The proof is completely analogous to that of Lemma 3.2 in \cite{BCW07}.
\end{proof}

\begin{proposition}
\label{prop:6}For every $\varepsilon >0$ there exists $v_{\varepsilon }\in 
\mathcal{Z}_{\varepsilon }$ such that $J_{\varepsilon }(v_{\varepsilon
})=d_{\varepsilon }$ and $v_{\varepsilon }$ is a sign changing solution of 
\emph{(\ref{eq:Peps})}. Moreover $d_{\varepsilon }\geq 2c_{\varepsilon }$.
\end{proposition}

\begin{proof}
If $(u_{k})$ is a minimizing sequence for $J_{\varepsilon }$ in $\mathcal{Z}%
_{\varepsilon }$ then, by Lemma \ref{lem:3}, we may assume that $\Vert
\nabla _{\varepsilon }J_{\varepsilon }(u_{k})\Vert _{\varepsilon
}\rightarrow 0$ as $k\rightarrow \infty .$ Since $M$ is compact, the
embedding $H_{g}^{1}(M)\hookrightarrow L_{g}^{p}(M)$ is compact, cf. \cite%
{H99}. Therefore, $J_{\varepsilon }:H_{g}^{1}(M)\rightarrow \mathbb{R}$
satisfies the Palais-Smale condition. So, passing to a subsequence, we have
that $u_{k}\rightarrow v_{\varepsilon }$ strongly in $H_{g}^{1}(M)$ and $%
J_{\varepsilon }(v_{\varepsilon })=d_{\varepsilon }.$ Since $\mathcal{Z}%
_{\varepsilon }$ is closed in $H_{g}^{1}(M)$, $v_{\varepsilon }\in \mathcal{Z%
}_{\varepsilon }$ and, since $\mathcal{Z}_{\varepsilon }$ is invariant under
the negative gradient flow, $v_{\varepsilon }$ is a stationary point for the
flow, i.e. a solution of (\ref{eq:Peps}). Finally, note that every sign
changing solution $v$ of (\ref{eq:Peps}) satisfies that $v^{\pm }\in 
\mathcal{N}_{\varepsilon }.$ So, by Lemma \ref{lem:lemma1}, $J_{\varepsilon
}(v)\geq 2c_{\varepsilon }.$ In particular, $J_{\varepsilon }(v_{\varepsilon
})=d_{\varepsilon }\geq 2c_{\varepsilon },$ as claimed.
\end{proof}

For each $u\in H_{g}^{1}(M)\smallsetminus \left\{ 0\right\} $ there exists a
unique positive number $t_{\varepsilon }(u)$ such that $t_{\varepsilon
}(u)u\in \mathcal{N}_{\varepsilon }.$ This number is given by 
\begin{equation}
t_{\varepsilon }^{p-2}(u)=\frac{\Vert u\Vert _{\varepsilon }^{2}}{%
|u|_{p,\varepsilon }^{p}}.  \label{eq:4}
\end{equation}%
We consider the set 
\begin{equation*}
F_{\varepsilon }(M):=\left\{ (q_{1},q_{2})\in M\times M:\text{dist}%
_{g}(q_{1},q_{2})\geq 2\varepsilon R\right\} ,
\end{equation*}%
where dist$_{g}$ is the distance in $M,$ and define $\iota _{\varepsilon
}:F_{\varepsilon }(M)\rightarrow H_{g}^{1}(M)$ by 
\begin{equation}
\iota _{\varepsilon }(q_{1},q_{2})=t_{\varepsilon }(W_{\varepsilon
,q_{1}})W_{\varepsilon ,q_{1}}-t_{\varepsilon }(W_{\varepsilon
,q_{2}})W_{\varepsilon ,q_{2}},  \label{iota}
\end{equation}%
where $W_{\varepsilon ,q}$ is the function defined in (\ref{eq:Weps}).

\begin{lemma}
\label{lem:7}For every $\varepsilon >0$ the map $\iota _{\varepsilon
}:F_{\varepsilon }(M)\rightarrow H_{g}^{1}(M)$ is continuous. For each $%
\delta >0$ there exists $\varepsilon _{0}>0$ such that, if $\varepsilon
<\varepsilon _{0}$ then 
\begin{equation*}
\iota _{\varepsilon }(q_{1},q_{2})\in J_{\varepsilon }^{2c_{\infty }+\delta
}\cap \mathcal{E}_{\varepsilon }\qquad \text{for all \ }(q_{1},q_{2})\in
\left( M\times M\right) _{\varepsilon }.
\end{equation*}
\end{lemma}

\begin{proof}
Since $W_{\varepsilon ,q_{1}}$ and $W_{\varepsilon ,q_{2}}$ have disjoint
supports, $\iota _{\varepsilon }(q_{1},q_{2})^{+}=t_{\varepsilon
}(W_{\varepsilon ,q_{1}})W_{\varepsilon ,q_{1}}\in \mathcal{N}_{\varepsilon
} $ and $\iota _{\varepsilon }(q_{1},q_{2})^{-}=-t_{\varepsilon
}(W_{\varepsilon ,q_{2}})W_{\varepsilon ,q_{2}}\in \mathcal{N}_{\varepsilon
}.$ Therefore, $\iota _{\varepsilon }(q_{1},q_{2})\in \mathcal{E}%
_{\varepsilon }$ for every $(q_{1},q_{2})\in F_{\varepsilon }(M).$ Moreover, 
\begin{equation*}
J_{\varepsilon }\left( \iota _{\varepsilon }(q_{1},q_{2})\right)
=J_{\varepsilon }\left( t_{\varepsilon }(W_{\varepsilon
,q_{1}})W_{\varepsilon ,q_{1}}\right) +J_{\varepsilon }\left( t_{\varepsilon
}(W_{\varepsilon ,q_{2}})W_{\varepsilon ,q_{2}}\right) .
\end{equation*}%
The result now follows from Proposition 4.2 in \cite{BBM07}.
\end{proof}

\begin{proposition}
\label{prop:8}${\lim_{\varepsilon \rightarrow 0}}d_{\varepsilon }=2c_{\infty
}$.
\end{proposition}

\begin{proof}
Let $\delta >0$. Arguing as in \cite{CCN97} we have that ${\inf_{\mathcal{E}%
_{\varepsilon }}J_{\varepsilon }}$ is attained and any minimizer of ${%
J_{\varepsilon }}$ on $\mathcal{E}_{\varepsilon }$ is a sign changing
solution of (\ref{eq:Peps}). Since every sign changing solution lies in $%
\mathcal{Z}_{\varepsilon },$ from Proposition \ref{prop:6} and Lemma \ref%
{lem:7} we conclude that 
\begin{equation}
2c_{\varepsilon }\leq d_{\varepsilon }\leq {\inf_{\mathcal{E}_{\varepsilon
}}J_{\varepsilon }}\leq 2c_{\infty }+\delta  \label{eq:6}
\end{equation}%
for $\varepsilon $ small enough. Passing to the limit as $\varepsilon
\rightarrow \infty $ and using Lemma \ref{lem:lemma1} we conclude that ${%
\lim_{\varepsilon \rightarrow 0}}d_{\varepsilon }=2c_{\infty }$, as claimed.
\end{proof}

In the following sections we shall use the following result proved by Weth
in \cite{W06}.

\begin{theorem}
\label{thm:tobias}There exists $\kappa _{0}\in (0,c_{\infty })$ such that $%
J_{\infty }(u)>2c_{\infty }+\kappa _{0}$ for every sign changing solution $u$
of the limit problem \emph{(\ref{eq:a1})}.
\end{theorem}

\section{Multiplicity of solutions}

\label{sec:Sezione-2}The goal of this section is to prove the first
assertion of Theorem \ref{thm:A}. More precisely, we prove the following
result.

\begin{theorem}
\label{thm:A1}There exists $\varepsilon _{0}>0$ such that for every $%
\varepsilon \in (0,\varepsilon _{0})$ problem \emph{(\ref{eq:Peps})} has at
least \emph{cupl}$\,C(M)$ pairs of sign changing solutions $\pm u$ with $%
J(u)\leq d_{\varepsilon }+\kappa _{0}$.
\end{theorem}

Here $\kappa _{0}$ is as in Theorem \ref{thm:tobias}. We start with some
lemmas.

\begin{lemma}
\label{lem:5}Let $u_{k}\in \mathcal{Z}_{\varepsilon _{k}}\cap J_{\varepsilon
_{k}}^{d_{\varepsilon _{k}}+\delta _{k}}$ where $\varepsilon _{k},\delta
_{k}>0$ are such that $\varepsilon _{k}\rightarrow 0$ and $\delta
_{k}\rightarrow 0$ as $k\rightarrow \infty .$ Then 
\begin{equation*}
\text{\emph{dist}}_{\varepsilon _{k}}(u_{k}^{\pm },\mathcal{N}_{\varepsilon
_{k}})\rightarrow 0\text{\qquad and\qquad }J_{\varepsilon _{k}}(u_{k}^{\pm
})\rightarrow c_{\infty }\text{\qquad as }k\rightarrow \infty .
\end{equation*}
\end{lemma}

\begin{proof}
By Lemma \ref{ekeland} we may assume without loss of generality that $\Vert
\nabla _{\varepsilon _{k}}J_{\varepsilon _{k}}(u_{k})\Vert _{\varepsilon
_{k}}\rightarrow 0.$ Since%
\begin{equation*}
\frac{p-2}{2p}\left\Vert u_{k}\right\Vert _{\varepsilon
_{k}}^{2}=J_{\varepsilon _{k}}(u_{k})-\frac{1}{p}J_{\varepsilon
_{k}}^{\prime }(u_{k})u_{k}\leq J_{\varepsilon _{k}}(u_{k})+\Vert \nabla
_{\varepsilon _{k}}J_{\varepsilon _{k}}(u_{k})\Vert _{\varepsilon
_{k}}\left\Vert u_{k}\right\Vert _{\varepsilon _{k}},
\end{equation*}%
we have that $\left\Vert u_{k}\right\Vert _{\varepsilon _{k}}$ is uniformly
bounded. Hence,%
\begin{equation*}
\left\vert J_{\varepsilon _{k}}^{\prime }(u_{k})u_{k}^{\pm }\right\vert
=\left\vert \left\Vert u_{k}^{\pm }\right\Vert _{\varepsilon
_{k}}^{2}-\left\vert u_{k}^{\pm }\right\vert _{p,\varepsilon
_{k}}^{p}\right\vert \rightarrow 0.
\end{equation*}%
Consequently, the number $t_{\varepsilon _{k}}(u_{k}^{\pm })$ defined by (%
\ref{eq:4}) tends to $1$ and, therefore, 
\begin{equation*}
\text{dist}_{\varepsilon _{k}}(u_{k}^{\pm },\mathcal{N}_{\varepsilon
_{k}})\leq \left\Vert u_{k}^{\pm }-t_{\varepsilon _{k}}(u_{k}^{\pm
})u_{k}^{\pm }\right\Vert _{\varepsilon _{k}}\rightarrow 0.
\end{equation*}%
This, together with Lemma \ref{lem:lemma1} and Proposition \ref{prop:8},
implies that%
\begin{eqnarray*}
2c_{\infty } &\leq &\lim_{k\rightarrow \infty }J_{\varepsilon
_{k}}(t_{\varepsilon _{k}}(u_{k}^{+})u_{k}^{+})+\lim_{k\rightarrow \infty
}J_{\varepsilon _{k}}(t_{\varepsilon _{k}}(u_{k}^{-})u_{k}^{-}) \\
&=&\lim_{k\rightarrow \infty }J_{\varepsilon
_{k}}(u_{k}^{+})+\lim_{k\rightarrow \infty }J_{\varepsilon
_{k}}(u_{k}^{-})=\lim_{k\rightarrow \infty }J_{\varepsilon
_{k}}(u_{k})=2c_{\infty }.
\end{eqnarray*}%
Therefore, $J_{\varepsilon _{k}}(u_{k}^{\pm })\rightarrow c_{\infty }.$
\end{proof}

\begin{lemma}
\label{lem:6}For each $\eta \in (0,1)$ there exist $\delta _{0}>0$ and $%
\varepsilon _{0}>0$ such that, for every $u\in \mathcal{Z}_{\varepsilon
}\cap J_{\varepsilon }^{d_{\varepsilon }+\delta }$ with $\varepsilon \in
(0,\varepsilon _{0})$ and $\delta \in (0,\delta _{0})$, we can find $%
q^{(1)}=q^{(1)}(u)$ and $q^{(2)}=q^{(2)}(u)$ in $M$ such that 
\begin{equation}
\frac{1}{\varepsilon ^{n}}\int_{B_{g}(q^{(1)},\varepsilon R)}|u^{+}|^{p}d\mu
_{g}>(1-\eta ){\frac{2p}{p-2}}c_{\infty }  \label{eq:3.14}
\end{equation}%
\begin{equation}
\frac{1}{\varepsilon ^{n}}\int_{B_{g}(q^{(2)},\varepsilon R)}|u^{-}|^{p}d\mu
_{g}>(1-\eta ){\frac{2p}{p-2}}c_{\infty }.  \label{eq:3.15}
\end{equation}
\end{lemma}

\begin{proof}
We prove (\ref{eq:3.14}). Arguing by contradiction we assume there exist $%
\eta \in (0,1)$, $\varepsilon _{k},\delta _{k}>0$ and $u_{k}\in \mathcal{Z}%
_{\varepsilon _{k}}\cap J_{\varepsilon _{k}}^{d_{\varepsilon _{k}}+\delta
_{k}}$ such that $\varepsilon _{k}\rightarrow 0$ and $\delta _{k}\rightarrow
0$ as $k\rightarrow \infty $ and 
\begin{equation}
\frac{1}{\varepsilon _{k}^{n}}\int_{B_{g}(q,\varepsilon
_{k}R)}|u_{k}^{+}|^{p}d\mu _{g}\leq (1-\eta ){\frac{2p}{p-2}}c_{\infty }%
\text{\qquad for all }q\in M.  \label{eq:3.16}
\end{equation}%
Then, by Lemma \ref{lem:5},%
\begin{equation}
{\lim_{k\rightarrow +\infty }\Vert u_{k}^{\pm }\Vert _{\varepsilon
_{k}}^{2}=\lim_{k\rightarrow +\infty }|u_{k}^{\pm }|_{p,\varepsilon
_{k}}^{p}=\frac{2p}{p-2}c_{\infty }.}  \label{eq:3.19}
\end{equation}%
We divide the proof into several steps.

\textbf{Step 1.} There exist $\vartheta >0$, $T>0,$ $k_{0}\in \mathbb{N}$
and, for each $k>k_{0},$ a point $q_{k}\in M$ such that 
\begin{equation}
\frac{1}{\varepsilon _{k}^{n}}\int_{B_{g}(q_{k},T\varepsilon
_{k})}|u_{k}^{+}|^{p}d\mu _{g}>\vartheta \text{\qquad for all }k>k_{0}.
\label{eq:3.20}
\end{equation}

\begin{proof}[Proof of Step 1]
For each $k$ large enough we choose a finite partition $\{M_{j}^{k}:j\in
\Lambda _{k}\}$ of $M$ such that $M_{j}^{k}$ is closed, $M_{i}^{k}\cap
M_{j}^{k}\subset \partial M_{i}^{k}\cap \partial M_{j}^{k}$ for any $i\neq j$%
, and there exist $T>1$ and points $q_{j}^{k}\in M_{j}^{k}$ satisfying 
\begin{equation*}
B_{g}(q_{j}^{k},\varepsilon _{k})\subset M_{j}^{k}\subset
B_{g}(q_{j}^{k},T\varepsilon _{k})\text{\qquad for all }j\in \Lambda _{k}
\end{equation*}%
and such that each point $x\in M$ is contained in at most $m$ balls $%
B_{g}(q_{j}^{k},T\varepsilon _{k})$, where the number $m$ does not depend on 
$k.$ We denote by $u_{k,j}^{+}$ the restriction of $u_{k}^{+}$ to the set $%
M_{j}^{k}$. Then 
\begin{equation}
\left\vert u_{k}^{+}\right\vert _{p,\varepsilon
_{k}}^{p}=\tsum_{j}\left\vert u_{k,j}^{+}\right\vert _{p,\varepsilon
_{k}}^{p}\leq \max_{j}\left\vert u_{k,j}^{+}\right\vert _{p,\varepsilon
_{k}}^{p-2}\tsum_{j}\left\vert u_{k,j}^{+}\right\vert _{p,\varepsilon
_{k}}^{2}.  \label{eq:3.21}
\end{equation}%
We take a smooth cut-off function $\chi _{k}$ such that $\chi _{k}(t)\equiv
1 $ if $0<t<\varepsilon _{k}$ and $\chi _{\varepsilon }(t)\equiv 0$ if $%
t>T\varepsilon _{k}$ and $\left\vert \chi _{\varepsilon }^{\prime
}\right\vert \leq \frac{1}{\alpha \varepsilon _{k}}$. Then 
\begin{equation*}
\tilde{u}_{k,j}(x):=u_{k}^{+}(x)\chi _{k}(|x-q_{j}^{k}|)\in H_{g}^{1}(M).
\end{equation*}%
Hence, 
\begin{eqnarray*}
\left\vert u_{k,j}^{+}\right\vert _{p,\varepsilon _{k}}^{2} &\leq
&\left\vert \tilde{u}_{k,j}\right\vert _{p,\varepsilon _{k}}^{2}\leq c\Vert 
\tilde{u}_{k,j}\Vert _{\varepsilon _{k}}^{2} \\
&=&c\Vert u_{k,j}^{+}\Vert _{\varepsilon _{k}}^{2}+c\Vert \tilde{u}%
_{k,j}\mid _{B_{g}(q_{j}^{k},T\varepsilon _{k})\smallsetminus
M_{j}^{k}}\Vert _{\varepsilon _{k}}^{2} \\
&\leq &c\Vert u_{k,j}^{+}\Vert _{\varepsilon _{k}}^{2}+c\left( \frac{2}{%
\alpha ^{2}}+1\right) \Vert u_{k}^{+}\mid _{B_{g}(q_{j}^{k},T\varepsilon
_{k})\smallsetminus M_{j}^{k}}\Vert _{\varepsilon _{k}}^{2}
\end{eqnarray*}%
and, therefore, 
\begin{equation}
\tsum_{j}\left\vert u_{k,j}^{+}\right\vert _{p,\varepsilon _{k}}^{2}\leq
c\Vert u_{k}^{+}\Vert _{\varepsilon _{k}}^{2}+c\left( \frac{2}{\alpha ^{2}}%
+1\right) m\Vert u_{k}^{+}\Vert _{\varepsilon _{k}}^{2}.  \label{eq:3.22}
\end{equation}%
Combining (\ref{eq:3.21}) and (\ref{eq:3.22}) we obtain 
\begin{equation}
\left\vert u_{k}^{+}\right\vert _{p,\varepsilon _{k}}^{p}\leq c\Vert
u_{k}^{+}\Vert _{\varepsilon _{k}}^{2}\max_{j}\left\vert
u_{k,j}^{+}\right\vert _{p,\varepsilon _{k}}^{p-2}.  \label{eq:3.23}
\end{equation}%
which together with (\ref{eq:3.19}) implies that there exist $\vartheta >0$
and, for each $k$ large enough, a $j\in \Lambda _{k}$ such that%
\begin{equation*}
\vartheta <\left\vert u_{k,j}^{+}\right\vert _{p,\varepsilon _{k}}^{p}=\frac{%
1}{\varepsilon _{k}^{n}}\int_{M_{j}^{k}}|u_{k}^{+}|^{p}d\mu _{g}\leq \frac{1%
}{\varepsilon _{k}^{n}}\int_{B_{g}(q_{j}^{k},T\varepsilon
_{k})}|u_{k}^{+}|^{p}d\mu _{g}.
\end{equation*}%
This proves Step 1.
\end{proof}

Since $M$ is compact, we may assume that the sequence $(q_{k})$ converges to
a point $q\in M$. We define $w_{k}\in H^{1}(\mathbb{R}^{n})$ by%
\begin{equation}
{w_{k}(z):=\chi \left( \varepsilon _{k}|z|\right) u_{k}(\exp
_{q_{k}}(\varepsilon _{k}z))},  \label{eq:3.24}
\end{equation}%
where $\chi $ is a smooth cut-off function such that $\chi (t)\equiv 1$ for $%
0\leq t\leq \frac{R}{2}$, $\chi (t)\equiv 0$ for $t\geq R$ and $\left\vert
\chi ^{\prime }(t)\right\vert \leq \frac{2}{R}$. By (\ref{eq:3.19}) the
sequence $(\Vert u_{k}\Vert _{\varepsilon _{k}})$ is bounded. Therefore, $%
(w_{k})$ is bounded in $H^{1}(\mathbb{R}^{n})$ (see Lemma 5.6 in \cite{BBM07}%
). So, up to a subsequence, $w_{k}\rightharpoonup w$ weakly in $H^{1}(%
\mathbb{R}^{n}),$ $w_{k}\rightarrow w$ a.e. in $\mathbb{R}^{n}$\ and $%
w_{k}\rightarrow w$ strongly in $L_{loc}^{p}(\mathbb{R}^{n})$. The following
statement holds true.

\textbf{Step 2.} $w\in H^{1}(\mathbb{R}^{n})$ is a weak solution of the
equation $-\Delta w+w=|w|^{p-2}w$ and $w^{+}\not\equiv 0$.

\begin{proof}[Proof of Step 2]
Following the argument in the proof of Lemma 5.7 in \cite{BBM07} one shows
that $w$ is a weak solution of $-\Delta w+w=|w|^{p-2}w.$ Next we show that $%
w^{+}\not\equiv 0$. By Step 1 we have that, for $k$ large, 
\begin{eqnarray*}
|w_{k}^{+}|_{L^{p}(B(0,T))}^{p} &=&\frac{1}{\varepsilon _{k}^{n}}%
\int_{B(0,\varepsilon _{k}T)}|u_{k}^{+}(\exp _{q_{k}}(y))|^{p}dy \\
&\geq &\frac{1}{2\varepsilon _{k}^{n}}\int_{B(0,\varepsilon
_{k}T)}|u_{k}^{+}(\exp _{q_{k}}(y))|^{p}|g_{q_{k}}(y)|^{1/2}dy \\
&=&\frac{1}{2\varepsilon _{k}^{n}}\int_{B_{g}(q_{k},\varepsilon
_{k}T)}|u_{k}^{+}(x)|^{p}d\mu _{g}\geq \frac{\vartheta }{2}.
\end{eqnarray*}%
Here we use the fact that $\lim_{k\rightarrow +\infty
}|g_{q_{k}}(y)|^{1/2}=|g_{q}(0)|^{1/2}=1$ because $q_{k}\rightarrow q$ and $%
\left\vert y\right\vert \leq \varepsilon _{k}T$, hence $|g_{q_{k}}(y)|^{1/2}%
\geq \frac{1}{2}$ for $k$ large. Since $w_{k}\rightarrow w$ strongly in $%
L_{loc}^{p}(\mathbb{R}^{n}),$ passing to the limit we obtain that 
\begin{equation*}
|w^{+}|_{L^{p}(B(0,T))}^{p}=\lim_{k\rightarrow +\infty
}|w_{k}^{+}|_{L^{p}(B(0,T))}^{p}\geq \frac{\vartheta }{2}>0.
\end{equation*}%
This proves Step 2.
\end{proof}

\textbf{Step 3.} $J_{\infty }(w)=c_{\infty }$ and $w>0.$

\begin{proof}[Proof of Step 3]
Fix $\tau >0$. Since $|g_{q_{k}}(\varepsilon _{k}z)|$ converges to $%
|g_{q}(0)|=1$ then, for any $a\in (0,1),$ one has that $|g_{q_{k}}(%
\varepsilon _{k}z)|^{-1/2}\leq (1-a)^{-1}$ for $k$ large enough and $z\in
B(0,\tau )$. Therefore, for $k$ large enough we have 
\begin{eqnarray}
\int_{B(0,\tau )}|w_{k}(z)|^{p}dz &=&\int_{B(0,\tau )}\chi ^{p}({\varepsilon
_{k}}z)|u_{k}(\exp _{q_{k}}(\varepsilon _{k}z))|^{p}dz  \notag \\
&\leq &\left( \frac{1}{1-a}\right) \frac{1}{\varepsilon _{k}^{n}}%
\int_{B(0,\tau \varepsilon _{k})}|u_{k}(\exp
_{q_{k}}(y))|^{p}|g_{q_{k}}(y)|^{\frac{1}{2}}dy  \notag \\
&\leq &\left( \frac{1}{1-a}\right) \frac{1}{\varepsilon _{k}^{n}}%
\int_{B_{g}(q_{k},\tau \varepsilon _{k})}|u_{k}(x)|^{p}d\mu _{g}\leq \frac{1%
}{1-a}|u_{k}|_{p,\varepsilon _{k}}^{p}.  \label{eq:3.25}
\end{eqnarray}%
Then, by (\ref{eq:3.19}), we have 
\begin{equation*}
\int_{B(0,\tau )}|w(z)|^{p}dz=\lim_{k\rightarrow +\infty }\int_{B(0,\tau
)}|w_{k}(z)|^{p}dz\leq \frac{2p}{p-2}2c_{\infty }.
\end{equation*}%
Hence, 
\begin{equation*}
\int_{\mathbb{R}^{n}}|w(z)|^{p}dz=\lim_{\tau \rightarrow +\infty
}\int_{B(0,\tau )}|w_{k}(z)|^{p}dz\leq \frac{2p}{p-2}2c_{\infty }.
\end{equation*}%
It follows that $J_{\infty }(w)=\frac{p-2}{2p}|w|_{L^{p}(\mathbb{R}%
^{n})}^{p}\leq 2c_{\infty }$. Theorem \ref{thm:tobias} implies that $w>0$
and $J_{\infty }(w)=c_{\infty }.$ This proves Step 3.
\end{proof}

We are now ready to prove Lemma \ref{lem:6}.

Fix $a\in (0,\eta ).$ Since $w_{k}^{+}$ converges to $w=w^{+}$ in $%
L_{loc}^{p}(\mathbb{R}^{n})$ and $J_{\infty }(w)=c_{\infty }$, there exists $%
\tau >0$ such that, for $k$ large enough, 
\begin{equation*}
\left( \frac{1-\eta }{1-a}\right) \frac{2p}{p-2}c_{\infty }<\int_{B(0,\tau
)}|w_{k}^{+}(z)|^{p}dz.
\end{equation*}%
On the other hand, arguing as in Step 3 and using (\ref{eq:3.16})\ we have
that, for $k$ large enough, 
\begin{eqnarray*}
\int_{B(0,\tau )}|w_{k}^{+}(z)|^{p}dz &\leq &\left( \frac{1}{1-a}\right) 
\frac{1}{\varepsilon _{k}^{n}}\int_{B(0,\tau \varepsilon
_{k})}|u_{k}^{+}(\exp _{q_{k}}(y))|^{p}|g_{q_{k}}(y)|^{\frac{1}{2}}dy \\
&=&\left( \frac{1}{1-a}\right) \frac{1}{\varepsilon _{k}^{n}}%
\int_{B_{g}(q_{k},\varepsilon _{k}\tau )}|u_{k}^{+}(x)|^{p}d\mu _{g}\leq
\left( \frac{1-\eta }{1-a}\right) \frac{2p}{p-2}c_{\infty }.
\end{eqnarray*}%
This is a contradiction.

The proof of (\ref{eq:3.15}) is similar. This concludes the proof of Lemma %
\ref{lem:6}.
\end{proof}

By Nash's embedding theorem \cite{N56} we may assume that $M$ is
isometrically embedded in some euclidean space $\mathbb{R}^{N}.$ We fix $r>0$
such that $V_{r}:=\{x\in \mathbb{R}^{N}:$ dist$(x,M)\leq r\}$ is a tubular
neighborhood of $M.$ For $x\in V_{r}$ let $\pi (x)\in M$ be the unique point
in $M$ such that $\left\vert x-\pi (x)\right\vert =$ dist$(x,M)\leq r.$ The
map $\pi :V_{r}\rightarrow M$ is smooth and it is the normal disk bundle of
the embedding $M\hookrightarrow \mathbb{R}^{N}.$ We consider the map $\beta :%
\mathcal{N}_{\varepsilon }\rightarrow \mathbb{R}^{N}$ given by%
\begin{equation*}
\beta (u):=\frac{\int_{M}x\left\vert u(x)\right\vert ^{p}d\mu _{g}}{%
\int_{M}\left\vert u(x)\right\vert ^{p}d\mu _{g}}.
\end{equation*}%
If $\beta (u)\in V_{r}$ we set%
\begin{equation*}
\beta _{M}(u):=\pi (\beta (u)).
\end{equation*}%
The following statement holds true.

\begin{proposition}
\label{prop:9}There exist $\delta _{0}>0$ and $\varepsilon _{0}>0$ such
that, for every $u\in \mathcal{Z}_{\varepsilon }\cap J_{\varepsilon
}^{d_{\varepsilon }+\delta }$ with $\varepsilon \in (0,\varepsilon _{0})$
and $\delta \in (0,\delta _{0}]$,%
\begin{equation}
\beta (u^{+})\in V_{r},\text{\qquad }\beta (u^{-})\in V_{r}\text{\qquad
and\qquad }\beta _{M}(u^{+})\neq \beta _{M}(u^{-}).  \label{beta}
\end{equation}
\end{proposition}

\begin{proof}
Arguing by contradiction, let $\varepsilon _{k},\delta _{k}>0$ and $u_{k}\in 
\mathcal{Z}_{\varepsilon _{k}}\cap J_{\varepsilon _{k}}^{d_{\varepsilon
_{k}}+\delta _{k}}$ be such that $\varepsilon _{k}\rightarrow 0$, $\delta
_{k}\rightarrow 0$ and, for each $k,$ one of the three statements in (\ref%
{beta}) is not true.

By Lemma \ref{ekeland} there exist $v_{k}\in \mathcal{Z}_{\varepsilon
_{k}}\cap J_{\varepsilon _{k}}^{d_{\varepsilon _{k}}+\delta _{k}}$ such that 
$\Vert u_{k}-v_{k}\Vert _{\varepsilon _{k}}\leq \sqrt{\delta _{k}}$ and $%
\Vert \nabla _{\varepsilon _{k}}J_{\varepsilon _{k}}(v)\Vert _{\varepsilon
_{k}}\leq \sqrt{\delta _{k}}$. Then, an easy estimate shows that%
\begin{equation}
\left\vert \beta (u_{k}^{\pm })-\beta (v_{k}^{\pm })\right\vert \rightarrow
0.  \label{beta0}
\end{equation}%
By Lemma \ref{lem:6}, after passing to a subsequence, there exist $%
q_{k}^{1},q_{k}^{2}\in M$ such that 
\begin{eqnarray}
\frac{2p}{p-2}(1-\frac{1}{k})c_{\infty } &<&\frac{1}{\varepsilon _{k}^{n}}%
\int_{B_{g}(q_{k}^{1},\varepsilon _{k}R)}|v_{k}^{+}|^{p}d\mu _{g}
\label{beta1} \\
&\leq &\frac{1}{\varepsilon _{k}^{n}}\int_{M}|v_{k}^{+}|^{p}d\mu _{g}\leq 
\frac{2p}{p-2}(1+\frac{1}{k})c_{\infty }  \notag
\end{eqnarray}%
and%
\begin{eqnarray}
\frac{2p}{p-2}(1-\frac{1}{k})c_{\infty } &<&\frac{1}{\varepsilon _{k}^{n}}%
\int_{B_{g}(q_{k}^{2},\varepsilon _{k}R)}|v_{k}^{-}|^{p}d\mu _{g}
\label{beta2} \\
&\leq &\frac{1}{\varepsilon _{k}^{n}}\int_{M}|v_{k}^{-}|^{p}d\mu _{g}\leq 
\frac{2p}{p-2}(1+\frac{1}{k})c_{\infty }.  \notag
\end{eqnarray}%
To simplify notation we write $\rho (v):=\frac{|v|^{p}}{\int_{M}|v|^{p}d\mu
_{g}}$. Then we have 
\begin{eqnarray*}
\left\vert \beta (v_{k}^{+})-q_{k}^{1}\right\vert &=&\left\vert
\int_{M}(x-q_{k}^{1})\rho (v_{k}^{+})d\mu _{g}\right\vert \\
&\leq &\left\vert \int_{B_{g}(q_{k}^{1},\varepsilon _{k}R)}(x-q_{k}^{1})\rho
(v_{k}^{+})d\mu _{g}\right\vert \\
&&+\left\vert \int_{M\smallsetminus B_{g}(q_{k}^{1},\varepsilon
_{k}R)}(x-q_{k}^{1})\rho (v_{k}^{+})d\mu _{g}\right\vert \\
&\leq &\varepsilon _{k}R+\frac{c}{k}
\end{eqnarray*}%
for some positive constant $c$. This inequality, together with (\ref{beta0}%
), implies that $\beta (u_{k}^{+})\in V_{r}$ for $k$ large enough.
Similarly, $\left\vert \beta (u_{k}^{-})-q_{k}^{2}\right\vert \rightarrow 0$
and $\beta (u_{k}^{-})\in V_{r}$ for $k$ large enough. Therefore $u_{k}$
must satisfy $\beta _{M}(u_{k}^{+})=\beta _{M}(u_{k}^{-}).$

Since $M$ is compact, after passing to a subsequence, we have that $%
q_{k}^{1}\rightarrow q$ and $q_{k}^{2}\rightarrow q$. The limit is the same
because $\beta _{M}(u_{k}^{+})=\beta _{M}(u_{k}^{-}).$

Let $\chi $ be a smooth cut-off function such that $\chi (z)\equiv 1$ for $%
\left\vert z\right\vert \leq \frac{T}{2}$ and $\chi (z)\equiv 0$ for $%
\left\vert z\right\vert \geq T$, $T\in (0,R).$ Set 
\begin{equation*}
w_{k}^{i}(z):=\chi (\varepsilon _{k}z)v_{k}(\exp _{q_{k}^{i}}(\varepsilon
_{k}z)),\text{\qquad }i=1,2.
\end{equation*}%
Then, up to a subsequence, $w_{k}^{i}\rightharpoonup w^{i}$ weakly in $H^{1}(%
\mathbb{R}^{n}),$ $w_{k}^{i}\rightarrow w^{i}$ a.e. in $\mathbb{R}^{n}$\ and 
$w_{k}^{i}\rightarrow w^{i}$ strongly in $L_{loc}^{p}(\mathbb{R}^{n})$.
Arguing as in the proof of Lemma \ref{lem:6} we conclude that $w^{i}$ solves 
$-\Delta w+w=|w|^{p-2}w$ and $J_{\infty }(w^{i})\leq 2c_{\infty }$.
Inequality (\ref{beta1}) implies that $(w^{1})^{+}\neq 0.$ Hence, $w^{1}>0.$
Similarly, $w^{2}<0.$

Next, we consider%
\begin{equation*}
w_{k}(z):=\chi (\varepsilon _{k}z)v_{k}(\exp _{q}(\varepsilon _{k}z)).
\end{equation*}%
Again, up to a subsequence, $w_{k}\rightharpoonup w$ weakly in $H^{1}(%
\mathbb{R}^{n}),$ $w_{k}\rightarrow w$ a.e. in $\mathbb{R}^{n}$\ and $%
w_{k}\rightarrow w$ strongly in $L_{loc}^{p}(\mathbb{R}^{n})$. For every $%
\varphi \in \mathcal{C}_{c}^{\infty }(\mathbb{R}^{n})$ and $k$ large enough
we have that%
\begin{eqnarray*}
\int_{\mathbb{R}^{n}}w_{k}(z)\varphi (z)dz &=&\int_{\mathbb{R}%
^{n}}w_{k}^{1}(\psi _{k}(z))\varphi (z)dz \\
&=&\int_{\mathbb{R}^{n}}w_{k}^{1}(y)\varphi (\psi _{k}^{-1}(z))\left\vert
\det \psi _{k}^{\prime }(z)\right\vert dz,
\end{eqnarray*}%
where $\psi _{k}(z):=\varepsilon _{k}^{-1}\exp _{q_{k}^{1}}^{-1}(\exp
_{q}(\varepsilon _{k}z)).$\ Passing to the limit as $k\rightarrow \infty $
we obtain that 
\begin{equation*}
\int_{\mathbb{R}^{n}}w(z)\varphi (z)dz=\int_{\mathbb{R}^{n}}w^{1}(z)\varphi
(z)dz\text{\qquad for every }\varphi \in \mathcal{C}_{c}^{\infty }(\mathbb{R}%
^{n}).
\end{equation*}%
Hence, $w=w^{1}>0.$ Similarly, we conclude that $w=w^{2}<0.$ This is a
contradiction.
\end{proof}

Fix $\delta _{0}>0$ and $\varepsilon _{0}>0$ as in Proposition \ref{prop:9}
and such that the map $\iota _{\varepsilon }:F_{\varepsilon }(M)\rightarrow
J_{\varepsilon }^{d_{\varepsilon }+\delta _{0}}\cap \mathcal{E}_{\varepsilon
}$ given by (\ref{iota}) is well defined for $\varepsilon \in (0,\varepsilon
_{0}),$ where%
\begin{equation*}
F_{\varepsilon }(M):=\left\{ (x,y)\in M\times M:\text{dist}_{g}(x,y)\geq
2\varepsilon R\right\} \subset F(M).
\end{equation*}%
Define $\theta _{\varepsilon }:\mathcal{Z}_{\varepsilon }\cap J_{\varepsilon
}^{d_{\varepsilon }+\delta _{0}}\rightarrow F(M)$ by%
\begin{equation*}
\theta _{\varepsilon }(u):=(\beta _{M}(u^{+}),\beta _{M}(u^{-})).
\end{equation*}%
The group $\mathbb{Z}/2:=\{1,-1\}$ acts on $F(M)$ by $-1\cdot (x,y):=(y,x).$
Its $\mathbb{Z}/2$-orbit space is the configuration space $C(M).$ Similarly,
we denote by $C_{\varepsilon }(M)$ the $\mathbb{Z}/2$-orbit space of $%
F_{\varepsilon }(M).$ On the other hand, $\mathbb{Z}/2$ acts on $\mathcal{E}%
_{\varepsilon }$ and $\mathcal{Z}_{\varepsilon }$\ by multiplication. These
actions are free, and the maps $\iota _{\varepsilon }$ and $\theta
_{\varepsilon }$ are $\mathbb{Z}/2$-equivariant, i.e. 
\begin{equation*}
\iota _{\varepsilon }(x,y)=-\iota _{\varepsilon }(y,x)\text{\qquad and\qquad 
}\theta _{\varepsilon }(-u)=-1\cdot \theta _{\varepsilon }(u).
\end{equation*}%
Hence, they induce maps between the $\mathbb{Z}/2$-orbit spaces. We write $(%
\mathcal{Z}_{\varepsilon }\cap J_{\varepsilon }^{d_{\varepsilon }+\delta
_{0}})/\left( \mathbb{Z}/2\right) $ for the $\mathbb{Z}/2$-orbit space of $%
\mathcal{Z}_{\varepsilon }\cap J_{\varepsilon }^{d_{\varepsilon }+\delta
_{0}}$ and put a hat over a function to denote the induced map between $%
\mathbb{Z}/2$-orbit spaces, e.g. we write 
\begin{equation*}
\widehat{\theta }_{\varepsilon }:(\mathcal{Z}_{\varepsilon }\cap
J_{\varepsilon }^{d_{\varepsilon }+\delta _{0}})/\left( \mathbb{Z}/2\right)
\rightarrow C(M)
\end{equation*}%
for the map induced by $\theta _{\varepsilon }.$ Let $\mathcal{\check{H}}$
be Alexander-Spanier or \v{C}ech cohomology with $\mathbb{Z}/2$%
-coefficients. Recall that it coincides with singular cohomology $\mathcal{H}%
^{\ast }$ on manifolds or, more generally, on ENRs. The following statement
holds true.

\begin{proposition}
\label{mono} There exists a homomorphism 
\begin{equation*}
\tau _{\varepsilon }:\mathcal{\check{H}}((\mathcal{Z}_{\varepsilon }\cap
J_{\varepsilon }^{d_{\varepsilon }+\delta _{0}})/\left( \mathbb{Z}/2\right)
)\rightarrow \mathcal{H}^{\ast }(C_{\varepsilon }(M))
\end{equation*}%
such that the composition 
\begin{equation*}
\tau _{\varepsilon }\circ \widehat{\theta }_{\varepsilon }^{\ast }:\mathcal{H%
}^{\ast }(C(M))\rightarrow \mathcal{H}^{\ast }(C_{\varepsilon }(M))
\end{equation*}%
is the homomorphism induced by the inclusion $C_{\varepsilon
}(M)\hookrightarrow C(M)$, which is an isomorphism for $\varepsilon $ small
enough.
\end{proposition}

\begin{proof}
The proof is analogous to that of Proposition 5.2 in \cite{BCW07}. The idea
is to express a certain subset of $\mathcal{Z}_{\varepsilon }\cap
J_{\varepsilon }^{d_{\varepsilon }+\delta _{0}}$ as a fixed point set and to
use Dold's fixed point transfer \cite{d} to define $\tau _{\varepsilon }$.
We outline the proof for the reader's convenience.

Define $\gamma :H_{g}^{1}(M)\rightarrow \mathbb{R}$ by 
\begin{equation*}
\gamma (u)=\left\{ 
\begin{array}{ll}
\frac{|u|_{p,\varepsilon }^{p}}{\Vert u\Vert _{\varepsilon }^{2}}-1 & \text{%
if }u\neq 0, \\ 
-1 & \text{if }u=0.%
\end{array}%
\right.
\end{equation*}%
This function is continuous and satisfies 
\begin{equation*}
\gamma (u^{+})=0=\gamma (u^{-})\quad \text{if and only if}\quad u\in {%
\mathcal{E}}_{\varepsilon }.
\end{equation*}%
For $u\in H_{g}^{1}(M)$ we denote by $e(u)\in \lbrack 0,\infty ]$ the
entrance time of $u$ into the set ${\mathcal{D}}_{\varepsilon }$. That is, 
\begin{equation*}
e(u):=\inf \{t\in (0,\infty ]:\varphi _{\varepsilon }(t,u)\in {\mathcal{D}}%
_{\varepsilon }\},
\end{equation*}%
where $\varphi _{\varepsilon }$ is the negative gradient flow of $%
J_{\varepsilon }.$ Since ${\mathcal{D}}_{\varepsilon }$ is strictly
positively invariant, the map $e:H_{g}^{1}(M)\rightarrow \lbrack 0,\infty ]$
is continuous. Moreover, $e(u)=\infty $ if and only if $u\in {\mathcal{Z}}%
_{\varepsilon }$. Consider the retraction 
\begin{equation*}
\varrho :H_{g}^{1}(M)\smallsetminus {\mathcal{Z}}_{\varepsilon }\rightarrow {%
\mathcal{D}}_{\varepsilon },\quad \varrho (u)=\varphi _{\varepsilon
}(e(u),u),
\end{equation*}%
and define $\psi :H_{g}^{1}(M)\rightarrow \mathbb{R}^{2}$ by 
\begin{equation*}
\psi (u)=\left\{ 
\begin{array}{ll}
0 & \text{if \ }u\in {\mathcal{Z}}_{\varepsilon }, \\ 
\frac{1}{1+e(u)^{p}}\left( \gamma (\varrho (u)^{+}),\gamma (\varrho
(u)^{-})\right) & \text{if \ }u\in H_{g}^{1}(M)\smallsetminus {\mathcal{Z}}%
_{\varepsilon }.%
\end{array}%
\right.
\end{equation*}%
$\psi $ is continuous and, since ${\mathcal{D}}_{\varepsilon }\cap {\mathcal{%
E}}_{\varepsilon }=\emptyset $, it satisfies 
\begin{equation}
\psi (u)=0\text{ if and only if }u\in {\mathcal{Z}}_{\varepsilon }.
\label{psicero}
\end{equation}%
We fix $\kappa >1$ such that 
\begin{equation}
J_{\varepsilon }\left( \kappa (\lambda \iota _{\varepsilon }(x,y)^{+}+\mu
\iota _{\varepsilon }(x,y)^{-})\right) \leq 0\quad \text{if }(x,y)\in
F_{\varepsilon }(M)\text{ and }\max \{\lambda ,\mu \}\geq 1,  \label{J0}
\end{equation}%
For $(x,y)\in F_{\varepsilon }(M)$ we define 
\begin{equation*}
g_{x,y}:[0,1]\times \lbrack 0,1]\rightarrow \mathbb{R}^{2},\quad
g_{x,y}(\lambda ,\mu )=\psi \left( \kappa (\lambda \iota _{\varepsilon
}(x,y)^{+}+\mu \iota _{\varepsilon }(x,y)^{-})\right) .
\end{equation*}%
It follows from (\ref{psicero}) that 
\begin{equation}
\kappa (\lambda \iota _{\varepsilon }(x,y)^{+}+\mu \iota _{\varepsilon
}(x,y)^{-})\in \mathcal{Z}_{\varepsilon }\cap J_{\varepsilon
}^{d_{\varepsilon }+\delta _{0}}\text{\qquad if }g_{x,y}(\lambda ,\mu )=0
\label{g=0}
\end{equation}%
and, using (\ref{J0}), we also conclude that $\kappa (\lambda \iota
_{\varepsilon }(x,y)^{+}+\mu \iota _{\varepsilon }(x,y)^{-})\in {\mathcal{D}}%
_{\varepsilon }$ if $(\lambda ,\mu )\in \partial \left( \lbrack
0,1]^{2}\right) .$ Therefore, 
\begin{equation*}
g_{x,y}(\lambda ,\mu )=\left( (\kappa \lambda )^{p-2}-1,(\kappa \mu
)^{p-2}-1\right) \text{\qquad if }(\lambda ,\mu )\in \partial \left( \lbrack
0,1]^{2}\right) .
\end{equation*}%
Next we define 
\begin{equation*}
f:F_{\varepsilon }(M)\times \lbrack 0,1]^{2}\rightarrow F_{\varepsilon
}(M)\times \mathbb{R}^{2},\quad f(x,y,\lambda ,\mu )=(x,y,(\lambda ,\mu
)-g_{x,y}(\lambda ,\mu )).
\end{equation*}%
This map is equivariant with respect to the $\mathbb{Z}/2$-action given by 
\begin{equation*}
-1\cdot (x,y,\lambda ,\mu ):=(y,x,\mu ,\lambda ).
\end{equation*}%
The projection $\pi :F_{\varepsilon }(M)\times \mathbb{R}^{2}\rightarrow
F_{\varepsilon }(M)$ is also equivariant and, since the action is free, the
induced map between the $\mathbb{Z}/2$-orbit spaces is a vector bundle. We
denote by \textrm{Fix}$(\widehat{f})$ the set of fixed points of $\widehat{f}
$. Following the proof of Lemma 5.4 in \cite{BCW07} one shows that they have
the following properties:

\begin{enumerate}
\item[(i)] $\widehat{\pi }\circ \widehat{f}=\widehat{\pi },$

\item[(ii)] $\text{\textrm{Fix}}(\widehat{f})\subset (F_{\varepsilon
}(M)\times (0,1)^{2})/\left( \mathbb{Z}/2\right) ,$

\item[(iii)] The fixed point transfer $\tau _{\widehat{f}}:H^{\ast }(\text{%
\textrm{Fix}}(\widehat{f}))\rightarrow H^{\ast }(C_{\varepsilon }(M))$
satisfies $\tau _{\widehat{f}}\circ \widehat{\pi }^{\ast }=id$.
\end{enumerate}

It follows from (\ref{g=0}) that the map 
\begin{equation*}
\iota :\text{\textrm{Fix}}(f)\rightarrow \mathcal{Z}_{\varepsilon }\cap
J_{\varepsilon }^{d_{\varepsilon }+\delta _{0}},\quad \iota (x,y,\lambda
,\mu )=\kappa (\lambda \iota _{\varepsilon }(x,y)^{+}+\mu \iota
_{\varepsilon }(x,y)^{-})
\end{equation*}%
is well defined. It is equivariant and satisfies $\theta _{\varepsilon
}\circ \iota =i\circ \pi |_{\text{\textrm{Fix}}(f)}$, where $%
i:F_{\varepsilon }(M)\hookrightarrow F(M)$ is the inclusion. We define $\tau
_{\varepsilon }$ to be the composition 
\begin{equation*}
\tau _{\varepsilon }:\mathcal{\check{H}}((\mathcal{Z}_{\varepsilon }\cap
J_{\varepsilon }^{d_{\varepsilon }+\delta _{0}})/\left( \mathbb{Z}/2\right) )%
\overset{\widehat{\iota }^{\ast }}{\longrightarrow }\mathcal{\check{H}}(%
\text{\textrm{Fix}}(\widehat{f}))\overset{\tau _{\widehat{f}}}{%
\longrightarrow }\mathcal{\check{H}}(C_{\varepsilon }(M)).
\end{equation*}%
Using property (iii) we obtain 
\begin{equation*}
\tau _{\varepsilon }\circ \widehat{\theta }_{\varepsilon }^{\ast }=\tau _{%
\widehat{f}}\circ \widehat{\iota }^{\ast }\circ \widehat{\theta }%
_{\varepsilon }^{\ast }=\tau _{\widehat{f}}\circ \widehat{\pi }^{\ast }\circ 
\widehat{i}^{\ast }=\widehat{i}^{\ast },
\end{equation*}%
as claimed.
\end{proof}

\bigskip

\noindent \textbf{Proof of Theorem \ref{thm:A1}.}\qquad Fix $\delta _{0}\in
(0,\kappa _{0})$ and $\varepsilon _{0}>0$ as above and such that the
inclusion $F_{\varepsilon }(M)\hookrightarrow F(M)$ is a homotopy
equivalence for all $\varepsilon \in (0,\varepsilon _{0})$. Arguing as in
section 5.2 of \cite{BCW07} we have that $J_{\varepsilon }$ has at least cupl%
$\left[ (\mathcal{Z}_{\varepsilon }\cap J_{\varepsilon }^{d_{\varepsilon
}+\delta _{0}})/\left( \mathbb{Z}/2\right) \right] $ sign changing solutions 
$u_{\varepsilon }$ with $J_{\varepsilon }(u_{\varepsilon })\leq
d_{\varepsilon }+\delta _{0}$. Proposition \ref{mono} yields%
\begin{equation*}
\text{cupl}\left[ (\mathcal{Z}_{\varepsilon }\cap J_{\varepsilon
}^{d_{\varepsilon }+\delta _{0}})/\left( \mathbb{Z}/2\right) \right] \geq 
\text{cupl}C(M),
\end{equation*}%
as claimed. \ \qed\noindent

\section{The shape of low energy nodal solutions}

\label{sec:Sezione-3}We shall prove the following result.

\begin{theorem}
\label{thm:shape}There exists $\varepsilon _{0}>0$ such that for every sign
changing solution $u_{\varepsilon }$ to problem \emph{(\ref{eq:Peps})} with $%
\varepsilon \in (0,\varepsilon _{0})$ and $J_{\varepsilon }(u_{\varepsilon
})\leq d_{\varepsilon }+\kappa _{0}$ the following statements hold true:

\begin{enumerate}
\item[\emph{(a)}] $u_{\varepsilon }$ has a unique local maximum point $%
Q_{\varepsilon }$ and a unique local minimum point $q_{\varepsilon }$ on $M$.

\item[\emph{(b)}] For any fixed $T>0,$ 
\begin{align*}
\lim_{\varepsilon \rightarrow 0}\Vert u_{\varepsilon }(\exp _{Q_{\varepsilon
}}(\varepsilon z))-U(z)\Vert _{\mathcal{C}^{2}(B(0,T))}& =0, \\
\lim_{\varepsilon \rightarrow 0}\Vert u_{\varepsilon }(\exp _{q_{\varepsilon
}}(\varepsilon z))+U(z)\Vert _{\mathcal{C}^{2}(B(0,T))}& =0.
\end{align*}

\item[\emph{(c)}] Moreover, 
\begin{align*}
{\sup_{\xi \in M\smallsetminus B_{g}(Q_{\varepsilon },\varepsilon
T)}u_{\varepsilon }(\xi )}& <ce^{-\mu T}+\sigma _{1}(\varepsilon ), \\
\inf_{\xi \in M\smallsetminus B_{g}(q_{\varepsilon },\varepsilon
T)}u_{\varepsilon }(\xi )& >{-ce^{-\mu T}+\sigma }_{2}{(\varepsilon ),}
\end{align*}%
for some positive constants $c,\mu $ and for some functions $\sigma
_{1},\sigma _{2}$ which go to zero as $\varepsilon $ goes to zero.

\item[\emph{(d)}] The function $\Phi _{\varepsilon }$ given by%
\begin{equation*}
u_{\varepsilon }=W_{\varepsilon ,Q_{\varepsilon }}-W_{\varepsilon
,q_{\varepsilon }}+\Phi _{\varepsilon }
\end{equation*}%
is such that $\Vert \Phi _{\varepsilon }\Vert _{\varepsilon }\rightarrow 0$
as $\varepsilon \rightarrow 0$.
\end{enumerate}
\end{theorem}

We split the proof into several lemmas. Note that, if $u_{\varepsilon }$ is
a sign changing solution to problem (\ref{eq:Peps}) then $u_{\varepsilon
}\in \mathcal{C}^{2}(M).$ Hence, it has a maximum point $Q_{\varepsilon }$
and a minimum point $q_{\varepsilon }$ on $M$. Moreover, $u_{\varepsilon
}(Q_{\varepsilon })>0$ and $u_{\varepsilon }(q_{\varepsilon })<0.$ The
following estimates hold true.

\begin{lemma}
\label{lem:10}If $u_{\varepsilon }$ is a sign changing solution to problem 
\emph{(\ref{eq:Peps})} and $Q_{\varepsilon }$ is a maximum point and $%
q_{\varepsilon }$ is a minimum point of $u_{\varepsilon }$ on $M,$ then 
\begin{equation*}
u_{\varepsilon }\left( Q_{\varepsilon }\right) \geq 1\text{\qquad and\qquad }%
u_{\varepsilon }\left( q_{\varepsilon }\right) \leq -1.
\end{equation*}
\end{lemma}

\begin{proof}
Expressing $u_{\varepsilon }$ in local normal coordinates around the point $%
Q_{\varepsilon }$ we get 
\begin{equation*}
\tilde{u}_{\varepsilon }(z):=u_{\varepsilon }\left( \exp _{Q_{\varepsilon
}}(z)\right) \text{\qquad for }\left\vert z\right\vert <R,\ z\in \mathbb{R}%
^{n}.
\end{equation*}%
Recall that in these coordinates we have 
\begin{equation}
\Delta _{g}v=\frac{1}{\sqrt{|g|}}\sum_{ij}\frac{\partial }{\partial z_{i}}%
\left( g^{ij}(z)\sqrt{|g|}\frac{\partial v}{\partial z_{j}}\right) \text{%
\qquad and\qquad }g^{ij}(0)=\delta _{ij},  \label{eq:9}
\end{equation}%
where $|g(z)|:=\det (g_{ij}(z))$ and $(g^{ij}(z))$ is the inverse matrix of $%
(g_{ij}(z))$.\ Hence,%
\begin{equation}
-\varepsilon ^{2}\sum_{i=1}^{n}\frac{\partial ^{2}}{\partial y_{i}^{2}}%
\tilde{u}_{\varepsilon }(0)+\tilde{u}_{\varepsilon }(0)=\left( \tilde{u}%
_{\varepsilon }(0)\right) ^{p-1}.  \label{eq:10}
\end{equation}%
Since $0$ a maximum point of $\tilde{u}_{\varepsilon }$ we get from (\ref%
{eq:10}) that $1\leq \tilde{u}_{\varepsilon }(0)=u_{\varepsilon }\left(
Q_{\varepsilon }\right) $. This proves the first inequality. The proof of
the second one is similar.
\end{proof}

In the following two lemmas we assume that $\varepsilon _{k}\in (0,1)$ is
such that $\varepsilon _{k}$ converges to $0$ and that$\ u_{\varepsilon
_{k}} $ is a sign changing solution to problem (\ref{eq:Peps}) with $%
\varepsilon :=\varepsilon _{k}$ which satisfies $J_{\varepsilon
_{k}}(u_{\varepsilon _{k}})\leq d_{\varepsilon _{k}}+\kappa _{0},$ where $%
\kappa _{0}$ is as in Theorem \ref{thm:tobias}. Note that%
\begin{equation*}
\frac{2p}{p-2}d_{\varepsilon _{k}}\leq \left\Vert u_{\varepsilon
_{k}}\right\Vert _{\varepsilon _{k}}^{2}=\frac{2p}{p-2}J_{\varepsilon
_{k}}(u_{\varepsilon _{k}})\leq \frac{2p}{p-2}\left( d_{\varepsilon
_{k}}+\kappa _{0}\right) .
\end{equation*}%
So, by Proposition \ref{prop:8}, there are constants $c_{1},c_{2}$ such
that, for $k$ large enough,%
\begin{equation}
0<c_{1}\leq \left\Vert u_{\varepsilon _{k}}\right\Vert _{\varepsilon
_{k}}^{2}\leq c_{2}<\infty .  \label{bounds}
\end{equation}%
Let$\ Q_{k}:=Q_{\varepsilon _{k}}$ be a maximum point and $%
q_{k}:=q_{\varepsilon _{k}}$ be a minimum point of $u_{\varepsilon _{k}}$ on 
$M,$ and set 
\begin{align*}
w_{k}^{1}(z)& :=u_{\varepsilon _{k}}\left( \exp _{Q_{k}}(\varepsilon
_{k}z)\right) \chi \left( {\varepsilon _{k}}\left\vert z\right\vert \right) 
\text{\qquad for }z\in \mathbb{R}^{n}, \\
w_{k}^{2}(z)& :=u_{\varepsilon _{k}}\left( \exp _{q_{k}}(\varepsilon
_{k}z)\right) \chi \left( {\varepsilon _{k}}\left\vert z\right\vert \right) 
\text{\qquad for }z\in \mathbb{R}^{n}, \\
\tilde{w}_{k}^{1}(z)& :=u_{\varepsilon _{k}}\left( \exp _{Q_{k}}(\varepsilon
_{k}z)\right) \text{\qquad for }z\in B\left( 0,\frac{R}{{\varepsilon _{k}}}%
\right) \subset \mathbb{R}^{n}, \\
\tilde{w}_{k}^{2}(z)& :=u_{\varepsilon _{k}}\left( \exp _{q_{k}}(\varepsilon
_{k}z)\right) \text{\qquad for }z\in B\left( 0,\frac{R}{{\varepsilon _{k}}}%
\right) \subset \mathbb{R}^{n},
\end{align*}%
where $\chi :\mathbb{R}^{+}\rightarrow \lbrack 0,1]$ is a $\mathcal{C}%
^{\infty }$ cut-off function such that $\chi (t)\equiv 1$ if $0\leq t\leq
R/2 $ and $\chi (t)\equiv 0$ if $R\leq t$. Note that $\tilde{w}%
_{k}^{1}(z)=w_{k}^{1}(z)$ and $\tilde{w}_{k}^{2}(z)=w_{k}^{2}(z)$ for $%
\left\vert {z}\right\vert <\frac{R}{2\varepsilon _{k}}$.

\begin{lemma}
\label{lem:11}There exist $w^{1},w^{2}\in H^{1}(\mathbb{R}^{n})$ such that,
after passing to a subsequence, ${w_{k}^{i}\rightharpoonup w^{i}}$ weakly in 
$H^{1}(\mathbb{R}^{n}),$ and ${w_{k}^{i}\rightarrow w^{i}}$ strongly in $%
L_{loc}^{p}(\mathbb{R}^{n})$\ and in $\mathcal{C}_{loc}^{2}(\mathbb{R}^{n})$
for $i=1,2.$ The functions $w^{1}$ and $w^{2}$ are nontrivial solutions of
the equation 
\begin{equation}
-\Delta w+w=|w|^{p-2}w  \label{limit}
\end{equation}%
such that $0$ is a maximum point of $w^{1}$ and a minimum point of $w^{2}.$
\end{lemma}

\begin{proof}
Arguing as in the proof of Lemma 5.6 in \cite{BBM07} one shows that $\Vert
w_{k}^{i}\Vert _{H^{1}(\mathbb{R}^{n})}\leq c\left\Vert u_{\varepsilon
_{k}}\right\Vert _{\varepsilon _{k}}$ for some positive constant $c$
independent of $k.$ It follows from (\ref{bounds}) that $(w_{k}^{i})$ is
bounded in $H^{1}(\mathbb{R}^{n}).$ Hence, there exists $w^{i}\in H^{1}(%
\mathbb{R}^{n})$ such that a subsequence of $(w_{k}^{i})$ converges to $%
w^{i} $ weakly in $H^{1}(\mathbb{R}^{n})$ and strongly in $L_{loc}^{p}(%
\mathbb{R}^{n})$. The fact that $w^{i}$ solves (\ref{limit}) can be proved
as in Lemma 5.7 of \cite{BBM07}. On the other hand, for $|z|<R/\varepsilon
_{k}$ the function $\tilde{w}_{k}^{i}$ satisfies the following equation 
\begin{equation*}
-\frac{1}{\sqrt{|g(\varepsilon _{k}z)|}}\sum_{ij}\frac{\partial }{\partial
z_{i}}\left( g^{ij}(\varepsilon _{k}z)\sqrt{|g(\varepsilon _{k}z)|}\frac{%
\partial \tilde{w}_{k}^{i}}{\partial z_{j}}\right) +\tilde{w}%
_{k}^{i}=\left\vert \tilde{w}_{k}^{i}\right\vert ^{p-2}\tilde{w}_{k}^{i}.
\end{equation*}%
Arguing as above, we have that the sequence $({\tilde{w}_{k}^{i})}$ is
bounded in $H^{1}(B(0,R/\varepsilon _{k}))$. By the Sobolev embedding
theorem and interior Schauder estimates in ${B\left( 0,\frac{R}{\varepsilon
_{k}}\right) }$ we get that $({\tilde{w}_{k}^{i})}$ is bounded in $\mathcal{C%
}^{2,\alpha }(B(0,R/\varepsilon _{k}))$ for some $\alpha \in (0,1)$. Then,
up to subsequence we have 
\begin{equation*}
\tilde{w}_{k}^{i}\rightarrow \tilde{w}^{i}\text{\qquad in }\mathcal{C}%
_{loc}^{2}(\mathbb{R}^{n}).
\end{equation*}%
Clearly, $\tilde{w}^{i}=w^{i}\in H^{1}(\mathbb{R}^{n})\cap \mathcal{C}^{2}(%
\mathbb{R}^{n})$. Since $w_{k}^{1}(0)=u_{\varepsilon _{k}}\left(
Q_{k}\right) \geq 1$ for any $k$, we have that $w^{1}\neq 0$. Similarly for $%
w^{2}$.
\end{proof}

\begin{lemma}
\label{lem:12}The functions $w^{1}$ and $w^{2}$ do not change sign.
Moreover, 
\begin{equation*}
w^{1}=-w^{2}=U.
\end{equation*}
\end{lemma}

\begin{proof}
By Theorem \ref{thm:tobias}, in order to prove that $w^{i}$ does not change
sign it suffices to show that 
\begin{equation}
J_{\infty }(w^{i})\leq 2c_{\infty }+\kappa _{0}.  \label{eq:12}
\end{equation}%
Since $\left\vert g(0)\right\vert =1$ we have that, for any $\delta >0,$
there exists $\rho >0$ such that $|g(y)|^{-1/2}<1+\delta $ for $\left\vert
y\right\vert <\rho $. Using this fact we obtain 
\begin{align*}
\int_{\mathbb{R}^{n}}\left\vert w^{1}(z)\right\vert ^{p}dz& \leq
\liminf_{k\rightarrow \infty }\int_{\mathbb{R}^{n}}\left\vert
w_{k}^{1}(z)\right\vert ^{p}dz \\
& =\liminf_{k\rightarrow \infty }\frac{1}{\varepsilon _{k}^{n}}
\int_{|y|<R\varepsilon_k}\left\vert u_{k}(\exp _{Q_{k}}(y)) \chi \left( 
\frac{|y|}{\varepsilon_k}\right) \right\vert ^{p}\frac{|g(y)|^{1/2}}{%
|g(y)|^{1/2}}dy \\
& \leq (1+\delta )\liminf_{k\rightarrow \infty }\frac{1}{\varepsilon _{k}^{n}%
}\int_{B_{g}(Q_{k},R\varepsilon_k)}\left\vert u_{\varepsilon _{k}}\left(
x\right) \right\vert ^{p}d\mu _{g} \\
& \leq (1+\delta )\liminf_{k\rightarrow \infty }\frac{1}{\varepsilon _{k}^{n}%
}\int_{M}\left\vert u_{\varepsilon _{k}}\left( x\right) \right\vert ^{p}d\mu
_{g}.
\end{align*}%
Multiplying by $\frac{p-2}{2p}$ and using Proposition \ref{prop:8}\ we
conclude that%
\begin{equation*}
J_{\infty }(w^{1})\leq (1+\delta )\liminf_{k\rightarrow \infty
}J_{\varepsilon _{k}}(u_{\varepsilon _{k}})\leq (1+\delta
)\liminf_{k\rightarrow \infty }\left( d_{\varepsilon _{k}}+\kappa
_{0}\right) =(1+\delta )\left( 2c_{\infty }+\kappa _{0}\right) .
\end{equation*}%
Since $\delta >0$ is arbitrary, this yields inequality (\ref{eq:12}). We
conclude that $w^{1}$ is a nontrivial solution to (\ref{limit}) which does
not change sign. Since $0$ is a maximum point of $w^{1}$ it follows that $%
w^{1}=U.$ The statements for $w^{2}$ are proved similarly.
\end{proof}

Lemmas \ref{lem:11} and \ref{lem:12} together imply that%
\begin{equation}
\tilde{w}_{k}^{1}\rightarrow U\text{\qquad and\qquad }\tilde{w}%
_{k}^{2}\rightarrow -U\text{\qquad in }\mathcal{C}_{loc}^{2}(\mathbb{R}^{n}).
\label{C2conv}
\end{equation}

\begin{lemma}
\label{lem:15}If $u_{\varepsilon }$ is a sign changing solution to problem 
\emph{(\ref{eq:Peps})} which satisfies $J_{\varepsilon }(u_{\varepsilon
})\leq d_{\varepsilon }+\kappa _{0}$ then, for $\varepsilon $ small enough, $%
u_{\varepsilon }$ has a unique local maximum point $Q_{\varepsilon }$ and a
unique local minimum point $q_{\varepsilon }$ on $M.$
\end{lemma}

\begin{proof}
Arguing by contradiction we assume there is a sequence $(\varepsilon _{k})$
which goes to zero and, for each $k,$ a sign changing solution $%
u_{\varepsilon _{k}}$ of problem (\ref{eq:Peps}) which satisfies $%
J_{\varepsilon _{k}}(u_{\varepsilon _{k}})\leq d_{\varepsilon _{k}}+\kappa
_{0}$ and has three local extrema $q_{k}^{1},$ $q_{k}^{2}$ and $q_{k}^{3}$.
As before, we set 
\begin{equation*}
\tilde{w}_{k}^{i}(z):=u_{\varepsilon _{k}}(\exp _{q_{k}^{i}}(\varepsilon
_{k}z))\text{\qquad for }z\in B\left( 0,\frac{R}{\varepsilon_k}\right) ,%
\text{ }i=1,2,3.
\end{equation*}%
Fix $\kappa _{1}\in (\kappa _{0},c_{\infty })$ and choose $T>0$ such that 
\begin{equation}
\frac{p-2}{2p}\int_{B(0,T)}|U(z)|^{p}dz>\frac{2c_{\infty }+\kappa _{1}}{3}.
\label{eq:13}
\end{equation}%
It follows from (\ref{C2conv}) that, for $k$ large enough, $\tilde{w}%
_{k}^{i} $ has a unique local extremum point at $0$ in $B(0,2T).$ Hence, $%
u_{\varepsilon _{k}}$ has a unique local extremum point at $q_{k}^{i}$ in $%
B_{g}(q_{k}^{i},2\varepsilon _{k}T),$ for each $i=1,2,3.$ In particular, $%
B_{g}(q_{k}^{i},\varepsilon _{k}T)\cap B_{g}(q_{k}^{j},\varepsilon
_{k}T)=\emptyset $ if $i\neq j.$ On the other hand, since $\left\vert
g(0)\right\vert =1$, for any $\delta >0$ and $k$ large enough we have 
\begin{equation}
c_{0}:=\frac{2c_{\infty }+\kappa _{0}}{2c_{\infty }+\kappa _{1}}%
<|g(\varepsilon _{k}z)|^{\frac{1}{2}}\text{\qquad for }\left\vert
z\right\vert <T\text{ and }k\text{ large enough. }  \label{eq:14}
\end{equation}%
Therefore, for all sufficiently large $k$ we obtain 
\begin{align}
\frac{1}{\varepsilon _{k}^{n}}\int_{M}|u_{\varepsilon _{k}}|^{p}d\mu _{g}&
\geq \frac{1}{\varepsilon _{k}^{n}}\sum\limits_{i=1}^{3}%
\int_{B_{g}(q_{k}^{i},\varepsilon _{k}T)}|u_{\varepsilon _{k}}|^{p}d\mu _{g}
\notag \\
& =\sum\limits_{i=1}^{3}\int_{B(0,T)}|\tilde{w}_{k}^{i}(z)|^{p}|g(%
\varepsilon _{k}z)|^{\frac{1}{2}}dz  \notag \\
& \geq c_{0}\sum\limits_{i=1}^{3}\int_{B(0,T)}|\tilde{w}_{k}^{i}(z)|^{p}dz.
\label{eq:15}
\end{align}%
Multiplying by $\frac{p-2}{2p}$ we conclude that%
\begin{equation*}
d_{\varepsilon _{k}}+\kappa _{0}\geq J_{\varepsilon _{k}}(u_{\varepsilon
_{k}})\geq c_{0}\sum\limits_{i=1}^{3}\frac{p-2}{2p}\int_{B(0,T)}|\tilde{w}%
_{k}^{i}(z)|^{p}dz.
\end{equation*}%
Passing to the limit as $k\rightarrow \infty ,$ Proposition \ref{prop:8}\
and Lemmas \ref{lem:11} and \ref{lem:12}\ yield%
\begin{align*}
2c_{\infty }+\kappa _{0}& =\lim_{k\rightarrow \infty }d_{\varepsilon
_{k}}+\kappa _{0}\geq c_{0}\sum\limits_{i=1}^{3}\frac{p-2}{2p}%
\lim_{k\rightarrow \infty }\int_{B(0,T)}|\tilde{w}_{k}^{i}(z)|^{p}dz \\
& =3c_{0}\frac{p-2}{2p}\int_{B(0,T)}|U(z)|^{p}dz>c_{0}(2c_{\infty }+\kappa
_{1})=2c_{\infty }+\kappa _{0}.
\end{align*}%
This is a contradiction.
\end{proof}

\begin{lemma}
\label{lem:16}Fix $T>0.$ If $u_{\varepsilon }$ is a sign changing solution
to problem \emph{(\ref{eq:Peps})} which satisfies $J_{\varepsilon
}(u_{\varepsilon })\leq d_{\varepsilon }+\kappa _{0}$ then, for $\varepsilon 
$ small enough, there are constants $c,\mu $ and functions $\sigma
_{1},\sigma _{2}$ such that 
\begin{align*}
\sup_{\xi \in M\smallsetminus B_{g}(Q_{\varepsilon },\varepsilon
T)}u_{\varepsilon }(\xi )& <ce^{-\mu T}+\sigma _{1}(\varepsilon ), \\
\inf_{\xi \in M\smallsetminus B_{g}(q_{\varepsilon },\varepsilon
T)}u_{\varepsilon }(\xi )& >-ce^{-\mu T}+\sigma _{2}(\varepsilon ),
\end{align*}%
and ${\ \lim_{\varepsilon \rightarrow 0}\sigma }_{i}{(\varepsilon )=0}$ for $%
i=1,2.$
\end{lemma}

\begin{proof}
Since for $\varepsilon $ small enough $u_{\varepsilon }$ has a unique local
maximum, the supremum of $u_{\varepsilon }$ on $M\smallsetminus
B_{g}(Q_{\varepsilon },\varepsilon T)$ is attained at a point of the
boundary $\partial B_{g}(Q_{\varepsilon },\varepsilon T)$. We consider the
function 
\begin{equation*}
\tilde{w}_{\varepsilon }^{1}(z)=u_{\varepsilon }\left( \exp _{Q_{\varepsilon
}}(\varepsilon z)\right) \text{\qquad for }\left\vert z\right\vert \leq T.
\end{equation*}%
Then, using the decay (\ref{asympU}) of $U$, for some constants $c,\mu >0$
we have 
\begin{eqnarray*}
\sup_{\xi \in M\smallsetminus B_{g}(Q_{\varepsilon },\varepsilon
T)}u_{\varepsilon }(\xi ) &\leq &\sup_{|z|=T}\left\vert \tilde{w}%
_{\varepsilon }^{1}(z)\right\vert \leq
\sup_{|z|=T}U(z)+\sup_{|z|=T}\left\vert \tilde{w}_{\varepsilon
}^{1}(z)-U(z)\right\vert \\
&\leq &ce^{-\mu T}+\sup_{|z|=T}\left\vert \tilde{w}_{\varepsilon
}^{1}(z)-U(z)\right\vert .
\end{eqnarray*}%
By(\ref{C2conv}) we have that 
\begin{equation*}
\sigma _{1}(\varepsilon ):=\sup_{|z|\leq T}\left\vert \tilde{w}_{\varepsilon
}^{1}(z)-U(z)\right\vert \rightarrow 0\text{\qquad as }\varepsilon
\rightarrow 0.
\end{equation*}%
This proves the first inequality. Analogously for the other one.
\end{proof}

\begin{lemma}
\label{lem:17}If $u_{\varepsilon }$ is a sign changing solution to problem 
\emph{(\ref{eq:Peps})} which satisfies $J_{\varepsilon }(u_{\varepsilon
})\leq d_{\varepsilon }+\kappa _{0}$ then%
\begin{equation}
\lim_{\varepsilon \rightarrow 0}J_{\varepsilon }(u_{\varepsilon
})=2c_{\infty },  \label{limenergy}
\end{equation}
and the function $\Phi _{\varepsilon }$ given by 
\begin{equation*}
u_{\varepsilon }=W_{\varepsilon ,Q_{\varepsilon }}-W_{\varepsilon
,q_{\varepsilon }}+\Phi _{\varepsilon }
\end{equation*}%
satisfies that $\Vert \Phi _{\varepsilon }\Vert _{\varepsilon }\rightarrow 0$
as $\varepsilon \rightarrow 0$.
\end{lemma}

\begin{proof}
For fixed $T>0$ we set $B_{\varepsilon T}^{1}:=B_{g}\left( Q_{\varepsilon
},\varepsilon T\right) ,$ $B_{\varepsilon T}^{2}:=B_{g}\left( q_{\varepsilon
},\varepsilon T\right) $ and $A_{\varepsilon T}:=M\smallsetminus
(B_{\varepsilon T}^{1}\cup B_{\varepsilon T}^{2}).$ Recall that for $%
\varepsilon $ small enough we have $2T\varepsilon <d_{g}(Q_{\varepsilon
},q_{\varepsilon })$. For $D\subset M$ set 
\begin{equation*}
\left\Vert v\right\Vert _{\varepsilon ,D}^{2}:=\frac{1}{\varepsilon ^{n}}%
\int_{D}(\varepsilon ^{2}\left\vert \nabla _{g}v\right\vert ^{2}+v^{2})d\mu
_{g},\text{\qquad }\left\vert v\right\vert _{p,\varepsilon ,D}^{p}:=\frac{1}{%
\varepsilon ^{n}}\int_{D}\left\vert v\right\vert ^{p}d\mu _{g}.
\end{equation*}%
Then, for $i=1,2$ we have%
\begin{align*}
\Vert u_{\varepsilon }\Vert _{\varepsilon ,B_{\varepsilon T}^{i}}^{2}& =%
\frac{1}{\varepsilon ^{n}}\int_{B_{\varepsilon T}^{i}}(\varepsilon
^{2}\left\vert \nabla _{g}u_{\varepsilon }\right\vert ^{2}+u_{\varepsilon
}^{2})d\mu _{g} \\
& =\int_{|z|<T}\left[ \sum_{jm}g^{jm}(\varepsilon z)\frac{\partial }{%
\partial z_{j}}\tilde{w}_{\varepsilon }^{i}(z)\frac{\partial }{\partial z_{m}%
}\tilde{w}_{\varepsilon }^{i}(z)+\left( \tilde{w}_{\varepsilon
}^{i}(z)\right) ^{2}\right] |g(\varepsilon z)|^{\frac{1}{2}}dz \\
& \rightarrow \int_{|z|<T}\left( |\nabla U(z)|^{2}+U(z)^{2}\right)
dz=:c_{T}\qquad \text{as \ }\varepsilon \rightarrow 0.
\end{align*}%
It follows that%
\begin{equation}
\Vert u_{\varepsilon }\Vert _{\varepsilon }^{2}-\Vert u_{\varepsilon }\Vert
_{\varepsilon ,A_{\varepsilon T}}^{2}\rightarrow 2c_{T}\qquad \text{as \ }%
\varepsilon \rightarrow 0.  \label{normonA}
\end{equation}%
Similarly,%
\begin{equation}
\left\vert u_{\varepsilon }\right\vert _{p,\varepsilon }^{p}-\left\vert
u_{\varepsilon }\right\vert _{p,\varepsilon ,A_{\varepsilon
T}}^{p}\rightarrow 2c_{p,T}\qquad \text{as \ }\varepsilon \rightarrow 0,
\label{normponA}
\end{equation}%
where 
\begin{equation*}
c_{p,T}:=\int_{|z|<T}U(z)^{p}dz.
\end{equation*}%
We write $\Phi _{\varepsilon }$ as 
\begin{align}
\Phi _{\varepsilon }& =\left( u_{\varepsilon }-W_{\varepsilon
,Q_{\varepsilon }}\right) \mid _{B_{\varepsilon T}^{1}}+\left(
u_{\varepsilon }-W_{\varepsilon ,q_{\varepsilon }}\right) \mid
_{B_{\varepsilon T}^{2}}  \label{eq:5circ} \\
& +u_{\varepsilon }\mid _{A_{\varepsilon T}}-W_{\varepsilon ,Q_{\varepsilon
}}\mid _{B_{g}(Q_{\varepsilon },\varepsilon R)\smallsetminus B_{\varepsilon
}^{1}}+W_{\varepsilon ,q_{\varepsilon }}\mid _{B_{g}(q_{\varepsilon
},\varepsilon R)\smallsetminus B_{\varepsilon }^{2}}.  \notag
\end{align}%
Using (\ref{C2conv}) we obtain 
\begin{align}
& \Vert u_{\varepsilon }-W_{\varepsilon ,Q_{\varepsilon }}\Vert
_{\varepsilon ,B_{\varepsilon T}^{1}}^{2}  \label{eq:6circ} \\
& =\int_{|z|<T}\left[ \sum_{jm}g^{jm}(\varepsilon z)\frac{\partial }{%
\partial z_{j}}\left( \tilde{w}_{\varepsilon }^{1}(z)-U(z)\chi (\varepsilon
z)\right) \frac{\partial }{\partial z_{m}}\left( \tilde{w}_{\varepsilon
}^{1}(z)-U(z)\chi (\varepsilon z)\right) \right.  \notag \\
& +\left. \left( \tilde{w}_{\varepsilon }^{1}(z)-U(z)\chi (\varepsilon
z)\right) ^{2}\right] |g(\varepsilon z)|^{\frac{1}{2}}dz\rightarrow 0\text{%
\qquad as }\varepsilon \rightarrow 0.  \notag
\end{align}%
Similarly, as $\varepsilon \rightarrow 0,$%
\begin{equation*}
\Vert u_{\varepsilon }-W_{\varepsilon ,q_{\varepsilon }}\Vert _{\varepsilon
,B_{\varepsilon T}^{2}}^{2}\rightarrow 0,\text{\qquad }\left\vert
u_{\varepsilon }-W_{\varepsilon ,Q_{\varepsilon }}\right\vert
_{p,\varepsilon ,B_{\varepsilon T}^{1}}^{p}\rightarrow 0,\text{\qquad }%
\left\vert u_{\varepsilon }-W_{\varepsilon ,q_{\varepsilon }}\right\vert
_{p,\varepsilon ,B_{\varepsilon T}^{2}}^{p}\rightarrow 0.
\end{equation*}%
Moreover, there is a constant $C$ such that, for all $\varepsilon \in (0,1),$
\begin{align}
& \Vert W_{\varepsilon ,Q_{\varepsilon }}\Vert _{\varepsilon
,B_{g}(Q_{\varepsilon },\varepsilon R)\smallsetminus B_{\varepsilon
}^{1}}^{2}  \label{eq:7circ} \\
& =\int_{T<|z|<\frac{R}{\varepsilon }}|g(\varepsilon z)|^{\frac{1}{2}}\left[
\sum_{jm}g^{jm}(\varepsilon z)\frac{\partial }{\partial z_{j}}\left(
U(z)\chi (\varepsilon z)\right) \frac{\partial }{\partial z_{m}}\left(
U(z)\chi (\varepsilon z)\right) +\left( U(z)\chi (\varepsilon z)\right) ^{2}%
\right] dz  \notag \\
& \leq C\int_{|z|>T}\left( |\nabla U(z)|^{2}+U^{2}(z)\right) dz,  \notag
\end{align}%
Similarly,%
\begin{equation}
\Vert W_{\varepsilon ,q_{\varepsilon }}\Vert _{\varepsilon
,B_{g}(q_{\varepsilon },\varepsilon R)\smallsetminus B_{\varepsilon
}^{2}}^{2},\text{ \ }\left\vert W_{\varepsilon ,Q_{\varepsilon }}\right\vert
_{p,\varepsilon ,B_{g}(Q_{\varepsilon },\varepsilon R)\smallsetminus
B_{\varepsilon }^{1}}^{p}\text{, \ }\left\vert W_{\varepsilon
,q_{\varepsilon }}\right\vert _{p,\varepsilon ,B_{g}(Q_{\varepsilon
},\varepsilon R)\smallsetminus B_{\varepsilon }^{2}}^{p},  \label{other}
\end{equation}%
are bounded above by a function of $T$ which goes to zero as $T\rightarrow
\infty \cdot $

To prove (\ref{limenergy}) we argue by contradiction. Assume there is a
sequence $\varepsilon _{k}\rightarrow 0$ such that $J_{\varepsilon
_{k}}(u_{\varepsilon _{k}})\rightarrow 2c_{\infty }+\kappa _{1}$ with $%
\kappa _{1}\in (0,\kappa _{0}].$ Then (\ref{normonA}) and (\ref{normponA})
imply that, as\ $k\rightarrow \infty ,$%
\begin{equation}
\Vert u_{\varepsilon _{k}}\Vert _{\varepsilon _{k},A_{\varepsilon
_{k}T}}^{2}\rightarrow \frac{2p}{p-2}(2c_{\infty }+\kappa
_{1})-2c_{T},\qquad \left\vert u_{\varepsilon _{k}}\right\vert
_{p,\varepsilon _{k},A_{\varepsilon _{k}T}}^{p}\rightarrow \frac{2p}{p-2}%
(2c_{\infty }+\kappa _{1})-2c_{p,T}.  \label{eq:8circ}
\end{equation}%
Let $\delta \in (0,1).$ Since $c_{T}\rightarrow \frac{2p}{p-2}c_{\infty },$ $%
c_{p,T}\rightarrow \frac{2p}{p-2}c_{\infty },$ and the right hand sides of (%
\ref{eq:7circ}) and of the similar inequalities for (\ref{other}) tend to
zero as $T\rightarrow \infty ,$ we may first choose $T>0$ and then choose $%
k_{0}=k_{0}(T)\in \mathbb{N}$ such that from (\ref{eq:5circ}), (\ref%
{eq:6circ}), (\ref{eq:7circ}) and (\ref{eq:8circ}) we obtain 
\begin{equation*}
\left\Vert \Phi _{\varepsilon _{k}}\right\Vert _{\varepsilon _{k}}^{2}\leq
(1+\delta )\frac{2p}{p-2}\kappa _{1}\qquad \text{and}\qquad \left\vert \Phi
_{\varepsilon _{k}}\right\vert _{p,\varepsilon _{k}}^{p}\geq (1-\delta )%
\frac{2p}{p-2}\kappa _{1}\qquad \forall k\geq k_{0}.
\end{equation*}%
Since we are assuming that $\kappa _{1}>0$ we have that $\Phi _{\varepsilon
_{k}}\neq 0$. Then, as in (\ref{eq:4}), $t_{\varepsilon _{k}}(\Phi
_{\varepsilon _{k}})\Phi _{\varepsilon _{k}}\in \mathcal{N}_{\varepsilon
_{k}}$ and%
\begin{equation*}
J_{\varepsilon _{k}}(t_{\varepsilon _{k}}(\Phi _{\varepsilon _{k}})\Phi
_{\varepsilon _{k}})=\frac{p-2}{2p}\left( \frac{\left\Vert \Phi
_{\varepsilon _{k}}\right\Vert _{\varepsilon _{k}}^{2}}{\left\vert \Phi
_{\varepsilon _{k}}\right\vert _{p,\varepsilon _{k}}^{2}}\right) ^{\frac{p}{%
p-2}}\leq \frac{(1+\delta )^{\frac{p}{p-2}}}{(1-\delta )^{\frac{2}{p-2}}}%
\kappa _{1}.
\end{equation*}%
Choosing $\delta $ small enough so that the right hand side of this
inequality is smaller than $c_{\infty }$ we obtain a contradiction to Lemma %
\ref{lem:lemma1}. This proves (\ref{limenergy}).

Identity (\ref{limenergy}) together with (\ref{normonA}) implies that 
\begin{equation*}
\Vert u_{\varepsilon }\Vert _{\varepsilon ,A_{\varepsilon T}}^{2}\rightarrow
2\left( \frac{2p}{p-2}c_{\infty }-c_{T}\right) \qquad \text{as \ }%
\varepsilon \rightarrow 0.
\end{equation*}%
Hence, for any $\eta >0$ we may choose $T>0$ large enough and $\varepsilon
_{0}=\varepsilon _{0}(T)>0$ so that from (\ref{eq:5circ}), (\ref{eq:6circ}),
(\ref{eq:7circ}) and (\ref{eq:8circ}) we obtain that 
\begin{equation*}
\Vert \Phi _{\varepsilon }\Vert _{\varepsilon }^{2}<\eta \qquad \text{for
every }\varepsilon \in (0,\varepsilon _{0}).
\end{equation*}%
This finishes de proof.
\end{proof}

\bigskip

\noindent \textbf{Proof of Theorem \ref{thm:shape}.}\qquad Parts \emph{(a)}, 
\emph{(b)}, \emph{(c)} and \emph{(d)} are given by Lemma \ref{lem:15},
statement (\ref{C2conv}) and Lemmas \ref{lem:16}\ and \ref{lem:17}\
respectively. \qed\noindent

\section{The cup-length of configuration spaces}

\label{sec:Sezione-5}Let $\pi :TM\rightarrow M$ be the tangent bundle of $M,$
whose fiber over $x$ is the tangent space $T_{x}M$ to $M$ at $x,$ and let $%
\pi :\mathbb{S}M\rightarrow M$ be its unit-sphere bundle. The group $\mathbb{%
Z}/2$ acts on $\mathbb{S}M$ by multiplication on each fiber, i.e. $-1\cdot
(x,z)=(x,-z)$ for every $x\in M$ and $z\in T_{x}M$ with $\left\vert
z\right\vert =1.$ We denote its $\mathbb{Z}/2$-orbit space by $\mathbb{P}M.$
Then, $\pi $ induces a map $\widehat{\pi }:\mathbb{P}M\rightarrow M$ which
is a fiber bundle with fiber the real projective space $\mathbb{R}P^{n-1}.$
The homomorphism $\theta :\mathcal{H}^{\ast }(\mathbb{R}P^{n-1})\rightarrow 
\mathcal{H}^{\ast }(\mathbb{P}M)$ which sends the generator $\omega \in 
\mathcal{H}^{1}(\mathbb{R}P^{n-1})$ to the first Stiefel-Whitney class $%
\widehat{\omega }\in \mathcal{H}^{1}(\mathbb{P}M)$ of the bundle $\mathbb{S}%
M\rightarrow \mathbb{P}M$ is a cohomology extension of the fiber, so the
Leray-Hirsch theorem \cite{s} provides an isomorphism%
\begin{equation}
\mathcal{H}^{\ast }(M)\otimes \mathcal{H}^{\ast }(\mathbb{R}P^{n-1})\cong 
\mathcal{H}^{\ast }(\mathbb{P}M)  \label{lh}
\end{equation}%
given by $\zeta \otimes \omega \mapsto \widehat{\pi }^{\ast }(\zeta
)\smallsmile \widehat{\omega }.$

If $V$ is a tubular neighborhood of $M$ in $\mathbb{R}^{N},$ we define a map 
$\alpha :\mathbb{S}M\rightarrow F(V)$ by%
\begin{equation*}
\alpha (x,z):=(\exp _{x}(\frac{R}{2}z),\exp _{x}(-\frac{R}{2}z)).
\end{equation*}%
The action of $\mathbb{Z}/2$\ on $F(V)$ is given by $-1\cdot (x,y)=(y,x).$
Therefore, $\alpha $ is $\mathbb{Z}/2$-equivariant, i.e. $\alpha
(x,-z)=-1\cdot \alpha (x,z),$ and it induces a map between the $\mathbb{Z}/2$%
-orbit spaces, which we denote by%
\begin{equation*}
\widehat{\alpha }:\mathbb{P}M\rightarrow C(V).
\end{equation*}

The cup-length of a map $f:X\rightarrow Y$ is the smallest integer $k\geq 1$
such that $f^{\ast }(\zeta _{1}\smallsmile \cdots \smallsmile \zeta _{k})=0$
for any $k$ cohomology classes $\zeta _{1},\dots ,\zeta _{k}\in \widetilde{%
\mathcal{H}}^{\ast }(Y)$. It is denoted cupl$(f).$ If $f$ is an inclusion $%
X\hookrightarrow Y$ we write cupl$_{Y}X:=$\ cupl$(f).$ It is easy to see
that cupl$(g\circ f)\leq \min \{$cupl$(f),$cupl$(g)\},$ cf. \cite{cp}. Since
the image of $\alpha $ is contained in $C(M),$ we conclude that%
\begin{equation}
\text{cupl}(\widehat{\alpha })\leq \text{ cupl}_{C(V)}C(M)\leq \text{ cupl}%
\,C(M).  \label{cuplalpha}
\end{equation}%
To prove Theorem \ref{thm:cuplength}\ we need the following lemma. Its proof
uses the Leray-Serre spectral sequence, which is treated for example in \cite%
{mc}.

\begin{lemma}
\label{lem:pullback}Let $V$ be a tubular neighborhood of $M$ in $\mathbb{R}%
^{N}$. Assume that $\mathcal{H}^{i}(M)=0$ for all $0<i<m$. Then, given a
cohomology class $\zeta \in \mathcal{H}^{m}(M),$ there exists $\widehat{%
\zeta }\in \mathcal{H}^{m}(C(V))$ such that 
\begin{equation*}
\widehat{\alpha }^{\ast }(\widehat{\zeta })=\widehat{\pi }^{\ast }(\zeta ).
\end{equation*}
\end{lemma}

\begin{proof}
Consider the diagram 
\begin{equation*}
\begin{array}{ccc}
\mathbb{P}M & \overset{\widehat{\alpha }}{\longrightarrow } & C(V) \\ 
\phi \uparrow \quad &  & \quad \uparrow \psi \\ 
\mathbb{S}M & \overset{\alpha }{\longrightarrow } & F(V) \\ 
\pi \downarrow \quad &  & \quad \downarrow \pi _{1} \\ 
M & \overset{i}{\hookrightarrow } & V%
\end{array}%
\end{equation*}%
where $\pi _{1}$ is the projection onto the first factor, and $\phi $ and $%
\psi $ are the obvious projections. Thus, $\widehat{\pi }\circ \phi =\pi .$
This diagram commutes up to homotopy. Since $M$ is a strong deformation
retract of $V,$ the inclusion $i$ induces an isomorphism in cohomology.
Hence, there exists $\overline{\zeta }\in \mathcal{H}^{m}(F(V))$ such that $%
\alpha ^{\ast }(\overline{\zeta })=\pi ^{\ast }(\zeta )\in \mathcal{H}^{m}(%
\mathbb{S}M).$ Next, we will show that $\overline{\zeta }=\psi ^{\ast }(%
\widetilde{\zeta })$ for some $\widetilde{\zeta }\in \mathcal{H}^{m}(C(V))$.

From the Thom-Gysin sequence of the sphere bundle $\pi :\mathbb{S}%
M\rightarrow M$ we obtain that $\pi ^{\ast }:\mathcal{H}^{i}(M)\rightarrow 
\mathcal{H}^{i}(\mathbb{S}M)$\ is an isomorphism for all $i<m+n-1$ and a
monomorphism for $i=m+n-1$. On the other hand, setting $D_{\varepsilon
}V:=\{(x,y)\in V\times V:\left\vert x-y\right\vert \leq \varepsilon \},$ and
using excision, homotopy invariance and the Thom isomorphism we obtain, for $%
\varepsilon $ small enough, 
\begin{equation*}
\mathcal{H}^{i}(V\times V,F(V))\cong \mathcal{H}^{i}(D_{\varepsilon
}V,D_{\varepsilon }V\smallsetminus D_{0}V)\cong \widetilde{\mathcal{H}}%
^{i-N}(V)\cong \widetilde{\mathcal{H}}^{i-N}(M).
\end{equation*}%
From the exact cohomology sequence of the pair $(V\times V,F(V))$ we deduce
that 
\begin{equation}
\mathcal{H}^{i}(F(V))\cong \mathcal{H}^{i}(V\times V)\cong \mathcal{H}%
^{i}(M\times M)\quad \text{for all }i<m+N-1.  \label{HFV}
\end{equation}

We consider the Leray-Serre spectral sequences of the Borel fibrations (for
the group $G=\mathbb{Z}/2$) 
\begin{equation*}
F(V)\underset{\mathbb{Z}/2}{\times }\mathbb{S}^{\infty }\rightarrow \mathbb{R%
}P^{\infty },\text{\qquad }\mathbb{S}M\underset{\mathbb{Z}/2}{\times }%
\mathbb{S}^{\infty }\rightarrow \mathbb{R}P^{\infty }.
\end{equation*}%
Since $\mathbb{Z}/2$ acts freely on $F(V)$ and on $\mathbb{S}M$, the
projections $F(V)\times \mathbb{S}^{\infty }\rightarrow F(V)$ and $\mathbb{S}%
M\times \mathbb{S}^{\infty }\rightarrow \mathbb{S}M$ induce homotopy
equivalences between the $\mathbb{Z}/2$-orbit spaces 
\begin{equation*}
F(V)\underset{\mathbb{Z}/2}{\times }\mathbb{S}^{\infty }\simeq C(V),\qquad 
\mathbb{S}M\underset{\mathbb{Z}/2}{\times }\mathbb{S}^{\infty }\simeq 
\mathbb{P}M.
\end{equation*}%
The map $\alpha :\mathbb{S}M\rightarrow F(V)$ induces a map of spectral
sequences 
\begin{equation*}
\alpha ^{\ast }:E_{r}^{p,q}\rightarrow \widetilde{E}_{r}^{p,q}.
\end{equation*}%
Their $E_{2}$-terms are 
\begin{equation*}
E_{2}^{p,q}=\mathcal{H}^{p}(\mathbb{R}P^{\infty };\mathcal{H}%
^{q}(F(V))),\qquad \widetilde{E}_{2}^{p,q}=\mathcal{H}^{p}(\mathbb{R}%
P^{\infty };\mathcal{H}^{q}(\mathbb{S}M)).
\end{equation*}%
Our assumptions on $\mathcal{H}^{\ast }(M)$ together with (\ref{HFV}) imply
that $\mathcal{H}^{q}(F(V))=0$ if $0<q<m$. Therefore, 
\begin{align*}
\mathcal{H}^{m}(F(V))& \cong E_{2}^{0,m}=\dots =E_{m+1}^{0,m}\text{,} \\
\mathcal{H}^{m+1}(\mathbb{R}P^{\infty })& \cong E_{2}^{m+1,0}=\dots
=E_{m+1}^{m+1,0}.
\end{align*}%
Since $\pi ^{\ast }(\zeta )$ is a permanent cycle and $\alpha ^{\ast }(%
\overline{\zeta })=\pi ^{\ast }(\zeta )$, we have that 
\begin{equation*}
\alpha ^{\ast }d_{m+1}(\overline{\zeta })=\widetilde{d}_{m+1}(\pi ^{\ast
}(\zeta ))=0.
\end{equation*}%
But $\alpha ^{\ast }$ is the identity on $\mathcal{H}^{m+1}(\mathbb{R}%
P^{\infty })\cong E_{2}^{m+1,0}=\widetilde{E}_{2}^{m+1,0}$. Hence $d_{m+1}(%
\overline{\zeta })=0$ and, therefore, $\overline{\zeta }$ is a permanent
cycle too. Thus, there exists $\widetilde{\zeta }\in \mathcal{H}^{m}(C(V))$
such that $\psi ^{\ast }(\widetilde{\zeta })=\overline{\zeta }.$

Our assumptions on $\mathcal{H}^{\ast }(M),$\ together with (\ref{lh}),
imply that $\mathcal{H}^{m}(\mathbb{P}M)\cong \mathcal{H}^{m}(M)\oplus 
\mathcal{H}^{m}(\mathbb{R}P^{n-1}).$\ Since $\phi ^{\ast }\widehat{\alpha }%
^{\ast }(\widetilde{\zeta })=\phi ^{\ast }\pi ^{\ast }(\zeta )$ we conclude
that $\widehat{\alpha }^{\ast }(\widetilde{\zeta })$ is either $\widehat{\pi 
}^{\ast }(\zeta )$ or $\widehat{\pi }^{\ast }(\zeta )+\widehat{\omega }^{m}.$
In the first case we set $\widehat{\zeta }:=\widetilde{\zeta }$ and in the
second case we set $\widehat{\zeta }:=\widetilde{\zeta }-\widetilde{\omega }%
^{m}$, where $\widetilde{\omega }\in \mathcal{H}^{1}(C(M))$ is the first
Stiefel-Whitney class of the bundle $F(V)\rightarrow C(V).$ Since $\widehat{%
\alpha }^{\ast }(\widetilde{\omega })=\widehat{\omega },$ we conclude that $%
\widehat{\alpha }^{\ast }(\widehat{\zeta })=\widehat{\pi }^{\ast }(\zeta ),$
as claimed.
\end{proof}

\bigskip

\noindent \textbf{Proof of Theorem \ref{thm:cuplength}.}\qquad By (\ref%
{cuplalpha}) it suffices to show that cupl$(\widehat{\alpha })\geq k+n.$ Let 
$\zeta _{1},\ldots ,\zeta _{k}\in \mathcal{H}^{m}(M)$ be such that $\zeta
_{1}\smallsmile \cdot \cdot \cdot \smallsmile \zeta _{k}\neq 0.$ Then (\ref%
{lh}) yields $\widehat{\pi }^{\ast }(\zeta _{1}\smallsmile \cdot \cdot \cdot
\smallsmile \zeta _{k})\smallsmile \widehat{\omega }^{n-1}\neq 0.$ By Lemma %
\ref{lem:pullback}\ there exist $\widehat{\zeta }_{1},\ldots ,\widehat{\zeta 
}_{k}\in \mathcal{H}^{m}(C(V))$ such that $\widehat{\alpha }^{\ast }(%
\widehat{\zeta }_{i})=\widehat{\pi }^{\ast }(\zeta _{i}).$ Therefore,%
\begin{equation*}
\widehat{\alpha }^{\ast }(\widehat{\zeta }_{1}\smallsmile \cdot \cdot \cdot
\smallsmile \widehat{\zeta }_{k}\smallsmile \widetilde{\omega }^{n-1})=%
\widehat{\pi }^{\ast }(\zeta _{1}\smallsmile \cdot \cdot \cdot \smallsmile
\zeta _{k})\smallsmile \widehat{\omega }^{n-1}\neq 0.
\end{equation*}%
It follows that cupl$(\widehat{\alpha })\geq k+n$. \qed

\end{document}